\documentclass[10pt]{article}
\usepackage{amsmath}
\usepackage{amsfonts}
\usepackage{amssymb}
\usepackage{graphicx}
\usepackage{mathrsfs}
\usepackage{xcolor}
\usepackage{bbm}
\usepackage{verbatim}
\usepackage{dsfont}
\usepackage{mathrsfs}
\usepackage[body={15.5cm,21cm}, top=3cm]{geometry}
 \usepackage{paralist}
\usepackage{indentfirst}
\allowdisplaybreaks[4]
\usepackage{hyperref}
\providecommand{\U}[1]{\protect \rule{.1in}{.1in}}
\numberwithin{equation}{section}
\newtheorem{Theorem}{Theorem}[section]

\newtheorem{lemma}[Theorem]{Lemma}

\newtheorem{remark}[Theorem]{Remark}

\newenvironment{proof}[1][Proof]{\noindent \textbf{#1.} }{\  \rule{0.5em}{0.5em}}
\allowdisplaybreaks
\begin{document}
\title{Extended mean field games with terminal constraint via decoupling fields}
\author{Tianjiao Hua \thanks{School of Mathematical Sciences, Shanghai Jiao Tong University, China (htj960127@sjtu.edu.cn)}
\and
Peng Luo \thanks{School of Mathematical Sciences, Shanghai Jiao Tong University, China (peng.luo@sjtu.edu.cn)}}
\maketitle
\begin{abstract}
We consider a class of extended mean field games with common noises, where there exists a strictly terminal constraint. We solve the problem by reducing it to an unconstrained control problem by adding a penalized term in the cost functional and then taking a limit. Using the stochastic maximum principle, we characterize the solution of the unconstrained control problem in terms of a conditional mean field forward–backward stochastic differential equation (FBSDE). We obtain the wellposedness results of the FBSDE and the monotonicity property of its decoupling field. Based on that, we solve the original constrained problem and characterize its solution in terms of a system of coupled conditional mean field FBSDE with a free backward part. In particular, we obtain the solvability of a new type of coupled conditional mean field FBSDEs.
\end{abstract}


\textbf{MSC-classification}: 93E20; 60H30; 49N70
\section{Introduction}

Mean field games (MFGs) are a powerful tool
to analyze strategic interactions in large populations when each individual player has negligible influence on the behavior of other players. Introduced independently by
Huang, Malham\'e, and Caines \cite{huang2006} and Lasry and Lions \cite{Lions2007}, MFGs have received
considerable attention in the probability and stochastic control literature in the last decade. Carmona and Delarue introduced a probabilistic approach to investigate MFGs (see e.g. \cite{carmona2013siam,MFGbook1,MFGbook2}). When the mean field interaction is not only via the distribution of the states but also controls, this MFG system is called extended mean field game, also known as mean field game of controls.  Extended mean field games have been successfully applied to many practical problems, ranging from optimal liquidation problem under market impact to models of optimal exploitation of exhaustible resources (see e.g. \cite{MFGbook1,carmona2021probabilistic,alasseur2020extended,cardaliaguet2018mean}). 

Different from most control problems, there exist terminal constraints in optimal liquidation models. Motivated by this, in this paper, we introduce a class of linear-quadratic extended MFGs with common noises and terminal state constraint, i.e., $X_{T} = 0$, where the state coefficients and the cost functional vary with the mean field term in a nonlinear way. The terminal constraint implies the singular terminal condition of the value function. The singularity becomes a major challenge when determining the value function and applying verification arguments, even without mean field terms. There exist many studies to overcome the challenges resulting from the terminal state constraint. The majority of the literature  considers finite approximations of the singular terminal value and then shows that the minimal solution to the value function with a singular terminal condition can be obtained by a monotone convergence argument (see e.g.,  \cite{2014Ankirchner,2015Graewe,2016kruse,Popier2006,Popier2007}).  Graewe et al. \cite{2018horst} introduced an approach by determining the precise asymptotic behavior of a potential solution to the HJB equation at the terminal time, which was further generalized in Graewe and Horst \cite{2017Horst}.

The situation is more complicated for MFGs. Indeed, MFGs with a terminal state constraint have been considered before in the literature by means of so-called mean field (game) planning problems (MFGPs) introduced by Lions. In these problems, the terminal state constraint is given by a target density of the state at the terminal time. Different from the literature on MFGPs (see, e.g., Achdou et al. \cite{2012Achdou}, Gomes and Seneci \cite{2018Gomes}, Porretta \cite{2014porretta}), we consider an MFG with a strict terminal state constraint by applying a probabilistic approach. 

In order to solve the constrained control problem, we first study a related control problem without the constraint $X_{T} = 0$, but with an penalizing term $\frac{1}{2}LX_{T}^{2}$ in the cost functional. Making use of the results in \cite{hua2023linearquadratic}, solving the unconstrained optimal control problem is reduced to solve a conditional mean field FBSDE and obtain the optimal control. More importantly, we can establish the regularity of the decoupling field $u^{L}$ of the conditional mean field FBSDE. The decoupling field admits an additional structure and enables us to pass to the limit $u^{\infty}$ when the constant $L$ of the penalty function converges to infinity, which we further use for solving the extended MFGs with terminal constraint. Furthermore, based on the convergence of $u^{L}$, we deduce that the associated solution processes $(X^{L},Y^{L},Z^{L},Z^{0,L})$ converge to a quadruplet of processes $(X^{\infty},Y^{\infty},Z^{\infty},Z^{0,\infty})$. We demonstrate that $(X^{\infty},Y^{\infty},Z^{\infty},Z^{0,\infty})$ can be uniquely characterized as the solution to a conditional mean field FBSDE under some assumptions, where initial and terminal conditions are imposed to the forward equation.

Indeed, Ankirchner et al. \cite{2020Fromm} also used the decoupling field method to solve a control problem with terminal state constraint. In contrast, we extend it to an extended MFG with common noise and the corresponding Hamiltonian system is a conditional mean field FBSDE. As a consequence, the decoupling field of the conditional mean field FBSDE is a function of not only the spatial variable but also measure variable. Thus,  this has brought additional difficulties to our analysis of its structure and limit properties. To the best of our knowledge, our work is the first one to use the decoupling field to tackle extended MFGs with terminal state constraint. Moreover, the convergence result can be viewed as a consistency result for both the unconstrained and the constrained problems.

A MFG problem with terminal state constraint was also considered in \cite{FuHorstliquidation2021}. In comparison, we consider a more general form of controlled state dynamics and take the mean field term of states into account in the state dynamics and cost functional. Moreover, different from \cite{FuHorstliquidation2021}, we allow the state coefficients and the cost functional vary with the mean field term in a nonlinear way. Thus, our Hamiltonian system is not a linear conditional mean field FBSDE and then the approach of using affine ansatz is no longer applicable. More importantly, we remove the limitation of weak interaction condition by adding some monotonicity condition. It is worth to mention that in our setting, our control problem satisfies the criterion named "absorption", which means a player drops out of the market when her position hits zero.  MFGs with absorption have recently been considered in \cite{graewe2022maximum,fu2024mean}. In particular, this restriction is reasonable for exhaustible resources, while also to some extent can be viewed as a no statistical arbitrage condition as stated in \cite{fu2024mean}. 

The rest of our paper is organized as follows. In section 2, we formulate the extended MFG with common noise and terminal constraint, give the assumptions and the main result. In section 3, we introduce an unconstrained MFG, analyze the property of the decoupling field of its solution, and then solve the constrained MFG by taking a limit. Moreover, we establish a wellposedness result for a new type of conditional mean field FBSDEs.

\textbf{Notations and Conventions.}
 Let $(\Omega, \mathcal{F}, \mathbb{F}, \mathbb{P})$ be a complete filtered probability space which can support two independent d-dimensional Brownian motions: $W$ and $W^0$.
For any filtration $\mathbb{G}$, we introduce the following spaces: 
$\beta \in$ $L_{\mathbb{G}}^2([0, T] ; \mathbb{R}^{n})$ if $\beta: \Omega \times[0, T] \rightarrow \mathbb{R}^{n}$ is a $\mathbb{G}$-progressively measurable process such that $\mathbb{E}\left[\int_0^T\left|\beta_t\right|^2 d t\right]$ $<\infty$; $\alpha \in S^{2}_{\mathbb{G}}([0,T];\mathbb{R}^{n})$ if $\alpha: \Omega \times[0, T] \rightarrow \mathbb{R}^{n}$ is a $\mathbb{G}$-progressively measurable process such that
$\mathbb{E}\left[ \sup\limits _{0 \leq t \leq T}|\alpha_{t}|^{2}\right] < \infty$. For any $\sigma$-field $\mathcal{G}$, we denote $\xi \in L_{\mathcal{G}}^2$ if $\xi: \Omega \rightarrow \mathbb{R}$ is a $\mathcal{G}$-measurable random variable such that $\mathbb{E}\left[|\xi|^2\right]<\infty$; $\xi\in L^{\infty}_{\mathcal{G}}$, if $\xi: \Omega \rightarrow \mathbb{R}$ is a $\mathcal{G}$-measurable random variable such that there exists a nonnegative constant $C$ such that $ \mathbb{P}\left[|\xi|\geq C\right]=0$.

Unless otherwise stated, all equalities and inequalities between random variables and processes will be understood in the $\mathbb{P}$-a.s. and $\mathbb{P} \otimes d t$-a.e. sense, respectively. $|\cdot|$ denotes the Euclidean norm. $C^1(\mathbb{R}^{n};\mathbb{R}^k)$ denotes the space of all $\mathbb{R}^k$-valued and continuous functions $f$ on $\mathbb{R}^{n}$ with continuous derivatives.
\section{Setting and Main result}
First, we consider a class of linear-quadratic (LQ) $N$-player games of controls with terminal constraint. For a given $T>0$, let $(\widetilde{\Omega}, \tilde{\mathcal{F}}, \widetilde{\mathbb{F}}, \widetilde{\mathbb{P}})$ be a complete filtered probability space which can support $N+1$ independent d-dimensional Brownian motions: $W^i, 1 \leq i \leq N$, and $W^0$. Here $W^i$ denotes the idiosyncratic noises for the $i$-th player and $W^0$ denotes the common noise for all the players. Let $\widetilde{\mathbb{F}}:=\left\{\widetilde{\mathcal{F}}_t\right\}_{0 \leq t \leq T}$ where $\widetilde{\mathcal{F}}_t:=\left(\vee_{i=1}^N \mathcal{F}_t^{W^i}\right) \vee \mathcal{F}_t^{W^0} \vee \tilde{\mathcal{F}}_0$ and let $\widetilde{\mathbb{P}}$ have no atom in $\tilde{\mathcal{F}}_0$.

For $1 \leq i \leq N$, let $\xi^i \in L_{\widetilde{\mathcal{F}}_0}^2$ be independent and identically distribution (i.i.d.) random variables. Denote $\boldsymbol{x}:=\left(x^1, \ldots, x^N\right)$ and $\boldsymbol{\alpha}:=\left(\alpha^1, \ldots, \alpha^N\right)$. The dynamic of $i$-th player's state process $x^i \in \mathbb{R}$ is
$$
\left\{\begin{aligned}
d x_t^i & =\left[A_t x_t^i+B_t \alpha_t^i+f\left(t, \nu_{\boldsymbol{x}_t}^{N, i}\right)+b\left(t, \mu_{\boldsymbol{\alpha}_t}^{N, i}\right)\right] d t, \\
x_0^i & =\xi^i,
\end{aligned}\right.
$$
where the control process $\alpha^{i} \in \widetilde{\mathcal{A}}^{0}(0, T):=\left\{\alpha | \alpha \in L_{\widetilde{\mathbb{F}}}^2([0, T] ; \mathbb{R}) \text{ such that } x^{i}_{T}=0 \right\}$ and
$$
A_t, B_t:[0, T] \rightarrow \mathbb{R}, \quad f, b:[0, T] \times \mathbb{R} \rightarrow \mathbb{R} .
$$
The interactions among players are via the average of all other players' states and controls
$$
\mu_{\boldsymbol{\alpha}}^{N, i}:=\frac{1}{N-1} \sum_{j \neq i} \alpha^j, \quad \nu_{\boldsymbol{x}}^{N, i}:=\frac{1}{N-1} \sum_{j \neq i} x^j.
$$
The optimization problem of player $i=1, \ldots, N$ is to minimize the cost functional
\begin{equation*}
\mathcal{J}^i\left(\alpha^i, \boldsymbol{\alpha}^{-i}\right):=\frac{1}{2} \mathbb{E}\int_0^T\left[Q_t\left(x_t^i+l\left(t, \nu_{\boldsymbol{x}_t}^{N, i}\right)\right)^2+R_t\left(\alpha_t^i+h\left(t, \mu_{\boldsymbol{\alpha}_t}^{N, i}\right)\right)^2\right] dt, 
\end{equation*}
where $\boldsymbol{\alpha}^{-i}=\left(\alpha^1, \ldots, \alpha^{i-1}, \alpha^{i+1}, \ldots, \alpha^N\right)$ denotes a strategy profile of other players excluding the $i$-th player and
$$
Q_t, R_t:[0, T] \rightarrow \mathbb{R}, \quad l, h:[0, T] \times \mathbb{R} \rightarrow \mathbb{R}.
$$
The above $N$-player game problem is to find the Nash equilibrium $\alpha^*$, that is to find a strategy profile $\alpha^*=\left(\alpha^{*, 1}, \ldots, \alpha^{*, N}\right)$ where $\alpha^{*, i} \in \widetilde{\mathcal{A}}^{0}(0, T), 1 \leq i \leq N$, such that
$$
\mathcal{J}^i\left(\alpha^{*, i}, \boldsymbol{\alpha}^{*,-i}\right)=\inf _{\alpha^i \in \widetilde{\mathcal{A}}^{0}(0, T)} \mathcal{J}^i\left(\alpha^i, \boldsymbol{\alpha}^{*,-i}\right) .
$$

Solving for a Nash equilibrium of an N-player game is challenging when N is large due to
the curse of dimensionality. Here we consider cases with terminal constraints, which will be more complex. So we formally take the limit as $N\rightarrow \infty$ and consider the limit
problem instead. The N-player games  are left for future work.

Let $(\Omega, \mathcal{F}, \mathbb{F}, \mathbb{P})$ be a complete filtered probability space which can support two independent d-dimensional Brownian motions: $W$ and $W^0$. Here $W$ denotes the idiosyncratic noise and $W^0$ denotes the common noise. We let $\mathbb{F}:=\left\{\mathcal{F}_t\right\}_{t \in[0, T]}$, where $\mathcal{F}_t:=$ $\mathcal{F}_t^W \vee \mathcal{F}_t^{W^0} \vee \mathcal{F}_0$, and let $\mathbb{P}$ have no atom in $\mathcal{F}_0$. We denote $\mathbb{F}^0:=\{\mathcal{F}_t^{W^0}\}_{t \in[0, T]}$,  $\mathcal{A}^{0}(0, T):=\left\{\alpha | \alpha \in L_{\mathbb{F}}^2([0, T] ; \mathbb{R}) \text{ such that } X_{T} = 0\right\}$. 

Then the extended MFG problem with terminal constraint associated with the N-player game is defined as follows and the functions appearing are as provided in the N-player game.\\
 \textbf{Problem (C-MF)} Find an optimal control  $\alpha^{*} \in \mathcal{A}^{0}(0, T)$ for the stochastic control problem 
   \begin{equation*}
   \left\{\begin{aligned}
        &\alpha^{*} \in \underset{\alpha \in \mathcal{A}^{0}(0, T)}{\operatorname{argmin}} \mathcal{J}(\alpha | \mu,\nu):=\frac{1}{2} \mathbb{E}\int_0^T\Big[Q_{t}\left(X_t^{ \alpha}+l\left(t,\nu_t\right)\right)^2+R_{t}\Big(\alpha_t+h\left(t,\mu_t\right)\Big)^2\Big]dt;\\
  & X_{t}^{\alpha} = \xi +\int_{0}^{t}\left(A_{s}X_{s}^{\alpha}+B_{s}\alpha_{s}+f(s,\nu_{s})+b(s,\mu_{s})\right)ds,\quad \xi \in L^{2}_{\mathcal{F}_{0}},\quad \mu,\nu\in L^{2}_{\mathbb{F}^{0}}([0,T];\mathbb{R});\\
    &\mu_{t} = \mathbb{E}[\alpha_{t}^{*}|\mathcal{F}_{t}^{W^{0}}], \quad \nu_{t} = \mathbb{E}[X^{\alpha^{*}}_{t}|\mathcal{F}^{W^{0}}_{t}].
\end{aligned}\right.
\end{equation*} 

We now state the main assumptions on the coefficients.\\
\textbf{Assumption (H).} Let $K$ be a positive constant.
 \begin{itemize}
    \item[(i)] The initial condition $\xi$ is nonnegative, and $\xi\in L^{\infty}_{\mathcal{F}_{0}}$.    \item[(ii)] The mappings $A_{t},B_{t},R_{t},Q_{t} :[0,T] \rightarrow \mathbb{R}$ are measurable  functions bounded by $K$. Moreover, $Q_{t}> 0$, $A_{t}\leq 0$, and there exists a positive constant $\delta$ such that for any $t\in[0,T]$, $R_{t}\geq \delta$, $|B_{t}|\geq \delta$.
\item[(iii)] The mappings $ f(t,x),\ b(t,x),\ l(t,x),\ h(t,x):[0,T]\times \mathbb{R}\rightarrow \mathbb{R}$ are measurable functions. We also assume that $h(t,0) =f(t,0) = b(t,0) = l(t,0) = 0$ for all $t\in [0,T]$.
\item[(iv)] $f(t,\cdot),\ b(t,\cdot),\ l(t,\cdot),\ h(t,\cdot) \in C^{1}(\mathbb{R};\mathbb{R})$ and their first derivatives are all bounded by $K$ for all $t\in[0,T]$.  Moreover, there exists positive constant $\varepsilon_{0}$ such that 
$$
\left|1+h^{\prime}(t,\cdot)\right| \geq \varepsilon_{0},
$$ 
and
\begin{equation} \label{(b4)}
f^{\prime}(t,\cdot) \leq 0, \quad l^{\prime}(t,\cdot)\geq 0,\quad
\frac{B_{t}R_{t}^{-1}(b^{\prime}(t,\cdot)-B_{t}h^{\prime}(t,\cdot))}{1+h^{\prime}(t,\cdot)} \geq 0.
 \end{equation}
\end{itemize} 

    We now state the main result on the solvability of problem (C-MF). The involved functions $\rho$ and $u^{\infty}$ will be given in the next section and $X^{\infty}$ is the controlled state process satisfied $X_{T}^{\infty} = 0$.
\begin{Theorem}\label{main result}
    Under assumption (H), the strategy $\alpha^{\infty}_{t}$  defined by
    \begin{equation*}
 \alpha_{t}^{\infty}:= -B_{t}R_{t}^{-1}u^{\infty}(t,X_{t}^{\infty},\mathbb{E}[X_{t}^{\infty}|\mathcal{F}_{t}^{W^{0}}])-h(t,\rho(t,-R_{t}^{-1}B_{t}u^{\infty}(t,X^{\infty}_{t},\mathbb{E}[X_{t}^{\infty}|\mathcal{F}_{t}^{W^{0}}]))),~t\in [0,T), \alpha^{\infty}_{T} := 0.
\end{equation*}
is an optimal control  for problem (C-MF) and satisfies $  
 \mathcal{J}\left(\alpha^{\infty}\right)<\infty$.
\end{Theorem}
\section{Solving problem (C-MF) via a penalization method}
In this section, in order to solve problem (C-MF), we use a penalization method. In detail, in subsection \ref{3.1}, we introduce an extended MFG problem without the constraint $X_{T}=0$, but with an additional term $\frac{1}{2}X_{T}^{2}$ in the cost functional penalizing the deviation of $X_{T}$ from zero. We obtain the optimal strategy of the unconstrained MFGs by solving the corresponding conditional mean field FBSDE. Then in subsection \ref{3.2}, we analyze the monotonicity property of the decoupling field of the conditional mean field FBSDE and then solve problem (C-MF) by letting $L\rightarrow \infty$. Finally, in subsection \ref{3.3}, we establish the existence of the solution of a new type of coupled conditional mean field FSDEs, where initial and terminal conditions are imposed on the forward equation.
\subsection{Unconstrained problem}\label{3.1}
In this subsection, we consider the following extended MFG problem with $L>0$:\\
   \textbf{Problem (P-MF) } Find an optimal control $\alpha^{L} \in \mathcal{A}(0, T):=\left\{\alpha | \alpha \in L_{\mathbb{F}}^2([0, T] ; \mathbb{R}) \right\}$ for the stochastic control problem 
   \begin{equation*}
   \left\{\begin{aligned}
        &\alpha^{L} \in \underset{\alpha \in \mathcal{A}(0, T)}{\operatorname{argmin}} \mathcal{J}^{L}(\alpha | \mu,\nu):=\frac{1}{2} \mathbb{E}\left\{ \int_0^T\Big[Q_{t}\left(X_t^{ \alpha}+l\left(t,\nu_t\right)\right)^2+R_{t}\left(\alpha_t+h\left(t,\mu_t\right)\right)^2\Big]dt+L(X_{T}^{\alpha})^{2}\right\};\\
  & X_{t}^{\alpha} = \xi +\int_{0}^{t}\left(A_{s}X_{s}^{\alpha}+B_{s}\alpha_{s}+f(s,\nu_{s})+b(s,\mu_{s})\right)ds,\quad \xi \in L^{2}_{\mathcal{F}_{0}},\quad \mu,\nu\in L^{2}_{\mathbb{F}^{0}}([0,T];\mathbb{R});\\
    &\mu_{t} = \mathbb{E}[\alpha_{t}^{L}|\mathcal{F}_{t}^{W^{0}}], \quad \nu_{t} = \mathbb{E}[X^{\alpha^{L}}_{t}|\mathcal{F}^{W^{0}}_{t}].
\end{aligned}\right.
\end{equation*} 

The problem (P-MF) has been investigated in \cite{hua2023linearquadratic} using the stochastic maximum principle. Here, we will directly list the main results, and the detailed proof can be found in \cite{hua2023linearquadratic}. Under assumption (H), for any $L>0$, the corresponding optimal control process is given by
\begin{equation*}
\alpha_t^{L}=-R_{t}^{-1} B_{t} Y_t^{L}-h\left(t,\rho(t,-R_{t}^{-1} B_{t} \mathbb{E}[Y_{t}^{L}|\mathcal{F}_{t}^{W^{0}}])\right),\label{optimal control}
\end{equation*}
where we use the inverse function theorem to drive the existence of a uniformly Lipschitz continuous function satisfying that
\begin{equation*}
\rho^{\prime}(t,\cdot) = \frac{1}{1+h^{\prime}(t,\cdot)}, \quad 
\left|\rho^{\prime}(t,\cdot)\right|=\left|\frac{1}{1+h^{\prime}(t,\cdot)}\right| \leq \frac{1}{\varepsilon_0},\quad t\in[0,T],
\end{equation*}
and $(X_{t}^{L},Y_{t}^{L})$ are the first two components of the unique solution of the following conditional mean field FBSDE:
\begin{equation}\label{conditional mfFBSDE}
\left\{\begin{aligned}
d X_t^{L} &=  {\left[A_t X_t^{L}-B_t^2 R_t^{-1} Y_t^{L}-B_t h\left(t, \rho(t,-R_t^{-1} B_t \mathbb{E}[Y_t^{L} |\mathcal{F}_t^{W^{0}}])\right)\right.} \\
& \left.\quad +f\left(t, \mathbb{E}[X_t^{L} |\mathcal{F}_t^{W^{0}}]\right)+b\left(t, \rho(t,-R_t^{-1} B_t \mathbb{E}[Y_t^{L}|\mathcal{F}_t^{W^{0}}])\right)\right] d t, \\
d Y_t^{L} &=  -\left[A_t Y_t^{L}+Q_tX_t^{L}+Q_t l(t, \mathbb{E}[X_t^{L} |\mathcal{F}_t^{W^{0}}])\right] d t+Z_t^{L} d W_t+Z_t^{0,L} d W_t^0, \\
X_0^{L} &= \xi, \quad Y_T^{L}=LX_T^{L}.
\end{aligned}\right.
\end{equation}
Moreover, for the convenience of subsequent analysis, we introduce the processes $P_{t}^{L},\nu_{t}^{L},\varphi_{t}^{L},\Gamma^{L}_{t}$, and $P_t^{L}$ satisfies the following Riccati equation
\begin{equation}\label{p equation}
\left\{
\begin{aligned}
    dP_{t}^{L}&= B^{2}_{t}R_{t}^{-1}(P^{L}_{t})^{2}-2A_{t}P^{L}_{t}-Q_{t},\\
    P_{T}^{L} &= L,
    \end{aligned}\right.
\end{equation}
and $(\nu_{t}^{L},\varphi_{t}^{L},\Gamma^{L}_{t})$ satisfy the following FBSDE
\begin{equation}\label{expectation FBSDE}
\left\{\begin{aligned}
d \nu_t^{L} & = {\left[\left(A_{t}-B_{t}^{2}R_{t}^{-1}P_{t}^{L}\right) \nu_t^{L}-B_{t}^2 R_{t}^{-1} \varphi_t^{L}-B_{t} h\left(t,\rho\left(t,-R_{t}^{-1} B\left(P_t^{L} \nu_t^{L}+\varphi_t^{L}\right)\right)\right)\right.} \\
& \quad\left.+f\left(t,\nu_t^{L}\right)+b\left(t,\rho\left(t,-R_{t}^{-1} B_{t}\left(P_t^{L} \nu_t^{L}+\varphi_t^{L}\right)\right)\right)\right] d t\\
d \varphi_t^{L} & =  -\left[\left(A_{t}-B_{t}^{2}R_{t}^{-1}P_{t}^{L}\right) \varphi_t^{L}+P_t^{L} f\left(t,\nu_t^{L}\right)+P_t^{L} b\left(t,\rho\left(t,-R^{-1}_{t} B_{t}\left(P_t^{L} \nu_t^{L}+\varphi_t^{L}\right)\right)\right)\right. \\
& \quad\left.+Q_{t} l\left(t,\nu_t^{L}\right)-P_t B_{t} h\left(t,\rho\left(t,-R_{t}^{-1} B_{t}\left(P_t \nu_t^{L}+\varphi_t^{L}\right)\right)\right)\right] d t+\Gamma_t^{L} d W_t^0, \\
\nu_0^{L} &=  \mathbb{E}[\xi], \quad \varphi_T^{L}=L\nu_T^{L}.
\end{aligned}\right.
\end{equation}

Now, we give the following theorem to summarize the wellposedness results for the above equations and give the optimal control of problem (P-MF). The detailed proof can be found in \cite{hua2023linearquadratic}.
\begin{Theorem}\label{LQ}
    Suppose assumption (H) holds.   
    \begin{itemize}
        \item [(i)] The Riccati equation \eqref{p equation}  admits a unique solution $P_{t}^{L}$ on $[0,T]$ such that for any $t\in[0,T]$,
 \begin{equation*}
    0 < P_{t}^{L} \leq C,
 \end{equation*}
 where $C$ is a constant depending only on $K,T,L$. 
 \item[(ii)] The FBSDE \eqref{expectation FBSDE} admits a unique solution $(\nu^{L},\varphi^{L},\Gamma^{L})$. Moreover, define the function $\Phi^{L}:[0,T]\times\mathbb{R}\rightarrow \mathbb{R}$ via
\begin{equation*}
    \Phi^{L}(t,\nu) = \varphi_{t}^{t,\nu},\label{phi decoupling}
\end{equation*}
where $\varphi^{t,\nu}_{t}$ is the second component of the unique solution of FBSDE \eqref{expectation FBSDE} with initial condition $(t,\nu)\in[0,T]\times \mathbb{R}$. Then it satisfies that
\begin{equation*}
    \left|\Phi^{L}(t,\nu_{1})-\Phi^{L}(t,\nu_{2})\right| \leq C\left|\nu_{1}-\nu_{2}\right|,\quad  \forall \nu_{1},\nu_{2} \in \mathbb{R},\quad \forall t\in[0,T],\label{phi uniform}
\end{equation*}
where C only depending on $K,T,L$. 
 \item[(iii)] The conditional mean field FBSDE \eqref{conditional mfFBSDE} admits a unique solution $(X^{L},Y^{L},Z^{L},Z^{0,L})$.
Moreover, the function $u^{L}:[0,T]\times\mathbb{R}\times \mathbb{R}\rightarrow \mathbb{R}$ defined by 
\begin{equation}\label{u_p_phi}
   u^{L}(t,x,\nu) :=  P^{L}_{t} x + \Phi^{L}(t,\nu)
\end{equation}
is the decoupling field of conditional mean field FBSDE \eqref{conditional mfFBSDE}.
\item [(iv)] The process $\alpha^{L}$ defined by
\begin{equation*}
    \alpha_{t}^{L}:=-R_{t}^{-1} B_{t} Y_t^{L}-h\left(t,\rho(t,-R_{t}^{-1} B_{t} \mathbb{E}[Y_{t}^{L}|\mathcal{F}_{t}^{W^{0}}])\right)
\end{equation*}
is the optimal control of problem (P-MF).
\end{itemize}
\end{Theorem}

In the following lemma, we give a property of the optimal state trajectory of problem (P-MF). We show that when starting with a nonnegative initial position, it cannot be optimal to choose control $\alpha$ such that the controlled state process is increasing or negative at some time point.
\begin{lemma}\label{x trajectory}
 Under assumption (H), if $\alpha \in \mathcal{A}(0,T) $ solves the problem (P-MF), then 
\begin{equation*}
   X_s^\alpha=\xi+\int_0^s \left[A_{r} X_r^{\alpha}+B_{r} \alpha_r+f\left(r,\mathbb{E}[X^{\alpha}_{r}|\mathcal{F}_{r}^{W^{0}}]\right)+b\left(r,\mathbb{E}[\alpha_{r}|\mathcal{F}_{r}^{W^{0}}]\right)\right] dr,\quad s \in[0, T],  
 \end{equation*}
is non-increasing and nonnegative.
\end{lemma}
\begin{proof}
For any $\beta \in \mathcal{A}(0,T)$, we denote
$$
\mathcal{J}(\beta):= \frac{1}{2}\int_0^T\left[Q_{t}\left(X_t^{ \beta}+l(t,\mathbb{E}[X^{\beta}_{t}|\mathcal{F}_{t}^{W^{0}}])\right)^2+R_{t}\left(\beta_t+h\left(t,\mathbb{E}[\beta_{t}|\mathcal{F}_{t}^{W^{0}}]\right)\right)^2\right] d t+\frac{1}{2}L\left(X_T^{\beta}\right)^2,
$$
where $X^{\beta}$ satisfies
\begin{equation*}
    X_t^\beta=\xi+\int_0^t \left[A_{r} X_r^{\beta}+B_{r} \beta_r+f\left(r,\mathbb{E}[X^{\beta}_{r}|\mathcal{F}_{r}^{W^{0}}]\right)+b\left(r,\mathbb{E}[\beta_{r}|\mathcal{F}_{r}^{W^{0}}]\right)\right] dr,\quad t\in[0, T].
\end{equation*}
\textbf{Step 1:} Let $\alpha \in \mathcal{A}(0,T)$ be optimal, we show that $X^\alpha_{t} $ is nonnegative for any $t\in [0.T]$.  To this end, define the measurable set $E:=\left\{\inf _{s \in[0, T]} X_s^\alpha<0\right\}$ and the stopping time \footnote{Here and in the sequel, we use the convention $\inf \varnothing=\infty$.}
\begin{equation*}
\sigma:=\sigma^\alpha:=\inf \left\{s \in[0, T] | X_s^{\alpha}<0\right\} \text {. }
\end{equation*}
Note that $E=\{\sigma<T\}$ and that $\sigma$ is equal to $\infty$ on $E^c$. Furthermore, $X_\sigma^\alpha=0$ on $E$. Now, define $\beta$ as the strategy which coincides with $\alpha$ on $[0, \sigma(\omega) \wedge T]$, but is 0 on the interval $(\sigma(\omega) \wedge$ $T, T]$. Notice that $\alpha=\beta$ on $E^c$. Now let $\omega \in E$. Then $X_\sigma^\beta(\omega)=0$. Since $f(t,0)=b(t,0) = 0$, for all $t\in[0,T]$, we obtain $X_s^\beta(\omega)=0$  for all $s \in[\sigma(\omega), T]$, but $X_s^\beta(\omega)=X_s^\alpha(\omega)$ for $s \in[0, \sigma(\omega)]$. Furthermore, for $s \in[\sigma(\omega), T]$,
\begin{equation*}
\begin{aligned}
   &\quad Q_{s}\left(X^{\alpha}_{s}+l(s,\mathbb{E}[X^{\alpha}_{s}|\mathcal{F}_{s}^{W^{0}}]\right)^{2}(\omega)+R_{s}\left(\alpha_{s}+h(s,\mathbb{E}[\alpha_{s}|\mathcal{F}_{s}^{W^{0}}])\right)^{2}(\omega)\\
   & \quad -Q_{s}\left(X^{\beta}_{s}+l(s,\mathbb{E}[X^{\beta}_{s}|\mathcal{F}_{s}^{W^{0}}]\right)^{2}(\omega)-R_{s}\left(\beta_{s}+h(s,\mathbb{E}[\beta_{s}|\mathcal{F}_{s}^{W^{0}}])\right)^{2}(\omega)\\
   &=Q_{s}\left(X^{\alpha}_{s}+l(s,\mathbb{E}[X^{\alpha}_{s}])|\mathcal{F}_{s}^{W^{0}}\right)^{2}(\omega)+R_{s}\left(\alpha_{s}+h(s,\mathbb{E}[\alpha_{s}|\mathcal{F}_{s}^{W^{0}}])\right)^{2}(\omega)\\
   &\geq 0.
\end{aligned}
\end{equation*}
Note that for $s>\sigma(\omega)$ in a sufficiently small neighborhood of $\sigma(\omega)$, the left-hand side of the previous inequality is strictly larger than zero. Moreover,
\begin{equation*}
    L(X_T^{\alpha})^2(\omega) \geq L(X_T^{\beta})^2(\omega) = 0.
    \end{equation*}
Therefore, $\mathcal{J}(\alpha)(\omega)>\mathcal{J}(\beta)(\omega)$ for all $\omega \in E$, which due to the optimality of $\alpha$ can only mean $\mathbb{P}(E)=0$. In other words, $X^\alpha$ is nonnegative a.s.\\
\textbf{Step 2:} Let $\alpha \in \mathcal{A}(0,T)$ be optimal, we show that $X^\alpha_{t}$ does not have any points of increase on $[0, T]$. To this end, we define the stopping time
\begin{equation*}
\rho:=\rho^\alpha:=\inf \left\{s \in[0, T] | X_s^\alpha>\inf _{r \in[0, s]} X_r^\alpha\right\},
\end{equation*}
Let $\gamma$ be the strategy which equals $\alpha$ on $[0,\rho\wedge T]$, but equals zero on $(\rho \wedge T, T]$. Let us define another stopping time $\tau$
\begin{equation*}
    \tau:=\tau^{\alpha,\gamma} :=\inf \left\{s\in[0,T]| s>\rho \text{ and } X_{s}^{\alpha} = X_{s}^{\gamma}\right\}.
\end{equation*}
Note that $(\rho \wedge T, \tau \wedge T]$ is empty and $\alpha=\gamma$ on
\begin{equation*}
\widehat{E}:=\left\{\sup _{s \in[0, T]}\left(X_s^\alpha-\inf _{r \in[0, s]} X_r^\alpha\right) \leq 0\right\},
\end{equation*}
the event where $X^\alpha$ is non-increasing. Now let $\omega \in \widehat{E}^c$. Then $\rho(\omega)<T$ and $(\rho(\omega), \tau(\omega) \wedge T]$ is nonempty with $\gamma(\omega)$ vanishing on this interval. For $s\in (\rho(\omega), \tau(\omega) \wedge T]$, we have
\begin{equation*}
    dX^{\gamma}_{s}(\omega) = A_{s} X^{\gamma}_{s}(\omega)+f(s,\mathbb{E}[X^{\gamma}_{s}(\omega)|\mathcal{F}_{s}^{{W^{0}}}]) ds
\end{equation*}
From step 1, since $\alpha$ is the optimal control, we have $X^{\alpha}_{\rho}(\omega)\geq 0$ and then we have $ X^{\gamma}_{\rho}(\omega)=X^{\alpha}_{\rho}(\omega)\geq 0$. Under the conditions that for all $t\in[0,T]$, $A(t)\leq 0$, $f^{\prime}(t,0)\leq 0$ and $f(t,0) = 0$, we have that for all $s \in(\rho(\omega), \tau(\omega) \wedge T]$, $ 0 \leq X^{\gamma}_{s}(\omega)\leq X^{\alpha}_{s}(\omega) $. \\
Let us adjust strategy $\gamma$ to $\tilde{\gamma}$ such that $\tilde{\gamma}  = \gamma$ on $[0,\tau(\omega)\wedge T]$ and $\tilde{\gamma} = \alpha$ on $(\tau(\omega)\wedge T,T]$. Together with the  condition $l^{\prime}(s,\cdot) \geq 0$, we obtain for all $s \in(\rho(\omega), \tau(\omega) \wedge T]$ :
\begin{equation*}
\begin{aligned}
   & \quad Q_{s}\left(X^{\alpha}_{s}+l(s,\mathbb{E}[X^{\alpha}_{s}|\mathcal{F}_{s}^{W^{0}}]\right)^{2}(\omega)+R_{s}\left(s,\alpha_{s}+h(s,\mathbb{E}[\alpha_{s}|\mathcal{F}_{s}^{W^{0}}])\right)^{2}(\omega)\\
   & \quad-Q_{s}\left(X^{\tilde{\gamma}}_{s}+l(s,\mathbb{E}[X^{\tilde{\gamma}}_{s}|\mathcal{F}_{s}^{W^{0}}]\right)^{2}(\omega)-R_{s}\left(\tilde{\gamma}_{s}+h(s,\mathbb{E}[\tilde{\gamma}_{s}|\mathcal{F}_{s}^{W^{0}}])\right)^{2}(\omega)\\
   & \geq Q_{s}\left(X^{\alpha}_{s}+l(s,\mathbb{E}[X^{\alpha}_{s}|\mathcal{F}_{s}^{W^{0}}]\right)^{2}(\omega)-Q_{s}\left(X^{\tilde{\gamma}}_{s}+l(s,\mathbb{E}[X^{\tilde{\gamma}}_{s}|\mathcal{F}_{s}^{W^{0}}]\right)^{2}(\omega)\\
   &\geq 0,
\end{aligned}
\end{equation*}
 Note that for $s>\rho(\omega)$ in a sufficiently small neighborhood of $\rho(\omega)$ we have a strict inequality. Moreover, 
\begin{equation*}
    L\left(X_T^{ \alpha}\right)^2(\omega) \geq L(X_T^{ \tilde{\gamma}})^2(\omega),
    \end{equation*}
    and hence $\mathcal{J}(\alpha)(\omega) > \mathcal{J}(\tilde{\gamma})(\omega)$. Due to the optimality of $\alpha$, we have $\mathbb{P}(\hat{E}^{c}) = 0$. Therefore, the state process of an optimal strategy is non-increasing.
\end{proof}
\begin{remark}\label{y remark}
In theorem \ref{LQ}, we characterize the optimal control in terms of the conditional mean field FBSDE \eqref{conditional mfFBSDE}, and the first component of the solution $X^{L}_{t}$ is the optimal state trajectory. Combined with lemma \ref{x trajectory}, we can obtain that $X^{L}_{t}$ is nonnegative and non-increasing with respect to $t$. Moreover, by the BSDE comparison theorem, we obtain that $Y_{t}^{L}\geq 0$, which also indicates that, the decoupling field of FBSDE \eqref{conditional mfFBSDE} $u^{L}\geq 0$.
\end{remark}
\subsection{Taking the limit}\label{3.2}
The objective of this subsection is to show that the optimal control $\alpha^L$ of problem (P-MF) converges to an admissible strategy $\alpha^{\infty} \in \mathcal{A}^0\left(0, T\right)$ when $L \rightarrow \infty$, which provides an optimal strategy for problem (C-MF). This can be done by first proving the convergence of the decoupling field $u^L(t,x,\nu)$ of the conditional mean field FBSDE \eqref{conditional mfFBSDE} to some limit $u^{\infty}(t,x,\nu)$ and then showing the convergence of $X_{t}^L$ to some limit $X^{\infty}_{t}$ satisfying $X_{T}^{\infty} = 0$. Finally it leads us to the limit $\alpha^{\infty} \in \mathcal{A}^0\left(0,T\right)$. 

To prove the pointwise convergence of $u^{L}(t,x,\nu)$ with respect to $L$, we first prove the uniform boundedness property of function $u^{L}$.
\begin{lemma}\label{uniformly bounded}
Under assumption (H), there exist positive constants $C_{1},C_{2},C_{3}$ and $C_{4}$, satisfying $C_{3}<C_{1}$, which only depend on $K$, $\delta$, $T$  such that for all $L>0$, $t\in[0,T)$ and $x>0$, $\nu>0$, it holds that
\begin{equation}\label{u1}
    \frac{C_{1}}{\frac{1}{L}+(T-t)}=:\lambda^{L}_{t}\leq P^{L}_{t} \leq \kappa_{t}:=C_{2}\left(1+\frac{1}{T-t}\right),
\end{equation}
and
\begin{equation}\label{u1u2}
\frac{C_{3}}{\frac{1}{L}+(T-t)}=:\tilde{\lambda}^{L}_{t} \leq P_{t}^{L}+\Phi_{\nu}^{L}(t,\nu) \leq \tilde{\kappa}_{t}:=C_{4}\left(1+\frac{1}{T-t}\right).
\end{equation}
Moreover, we have that for $t\in[0,T)$, $L>0$, $x>0$ and $\nu>0$, $u^L(t, x,\nu)$ is uniformly bounded independently of $L$.
\end{lemma}
\begin{proof} The proof is splitted into three steps.\\
\textbf{Step 1:} Let us consider the bound of $P_{t}^{L}$, which satisfies the following Riccati equation
\begin{equation*}
\left\{
    \begin{aligned}
    dP_{t}^{L}&= B^{2}_{t}R_{t}^{-1}(P^{L}_{t})^{2}-2A_{t}P_{t}^{L}-Q_{t}\\
    P_{T}^{L} &= L.
    \end{aligned}\right.
\end{equation*}
We now define a generator $H:[0,T]\times \mathbb{R} \rightarrow \mathbb{R}$ via 
\begin{equation*}
    H(t,p) := -B^{2}_{t}R_{t}^{-1}p^{2}+2A_{t}p+Q_{t}
\end{equation*}
such that $dP_{t}^{L} = -H(t,P_{t}^{L})dt,\ t\in [0,T]$. According to Theorem \ref{LQ}, $P_{t}^{L}> 0$. Furthermore, we define a generator $\hat{H}:  \mathbb{R} \rightarrow \mathbb{R}$ via
\begin{equation*}
 \hat{H}(p) =   -\frac{K^{2}}{\delta}p^{2}-2Kp 
\end{equation*}
Then, we obtain that $\hat{H}(p) \leq H(t,p)$, for $t\in [0,T]$ and $p>0$.
Define the deterministic, bounded and non-negative process $\hat{P}: [0,T] \rightarrow [0,\infty)$ via
\begin{equation*}
   \hat{P}_{t} : =  \left(\frac{e^{2K(T-t)}}{L}+\frac{K^{2}}{\delta}(T-t) \frac{e^{2K(T-t)}-1}{2K(T-t)}\right)^{-1}, \quad t \in[0, T],
\end{equation*}
where an expression of the form $\frac{e^x-1}{x}$ is to be replaced by $1=\lim _{x \rightarrow 0} \frac{e^x-1}{x}$ in case $x$ is zero. Observe that $\hat{P}$ solves the ODE  $d\hat{P}_t=-\hat{H}(\hat{P}_t) d t,\ t \in[0, T]$, with terminal condition $\hat{P}_T= L$. The comparison principle implies $ \hat{P}_{t} \leq P_{t}^{L} $, which in turn implies $\lambda^L \leq P_{t}^{L}$ for an appropriately chosen $C_1$ depending on $T,K,\delta$. This proves the lower bound in \eqref{u1}. \\
\indent We now prove the upper bound in \eqref{u1}. To this end, define another locally Lipschitz generator $\bar{H}: \mathbb{R} \rightarrow \mathbb{R}$ via
\begin{equation*}
\bar{H}(p):=K_1-K_2 p^2, 
\end{equation*}
where $K_1:=K+\frac{2K^{3}}{\delta^{2}}$ and $K_{2}:=\frac{\delta^{2}}{2K}$. It holds that
\begin{equation*}
2Ap = 2 \left(\frac{\sqrt{2R}}{B}A\right)\left(\frac{B}{\sqrt{2R}}p\right) \leq 2RB^{-2}A^{2}+\frac{1}{2}B^{2}R^{-1}p^{2},
\end{equation*}
then we have
\begin{equation*}
\begin{aligned}
 H(t,p)  &= -B^{2}_{t}R_{t}^{-1}p^{2}+2A_{t}p+Q_{t}\\
  &\leq -\frac{1}{2}B^{2}_{t}R_{t}^{-1}p^{2}+2R_{t}B_{t}^{-2}A_{t}^{2} +Q_{t}   \\
  &\leq \frac{\delta^{2}}{2K}p^{2}+K+\frac{2K^{3}}{\delta^{2}}\\
  &=\bar{H}(p).
\end{aligned} 
\end{equation*}
 Now define the deterministic and bounded process $\bar{P}:[0, T] \rightarrow \mathbb{R}$ via
\begin{equation*}
\bar{P}_t=  \sqrt{\frac{K_1}{K_2}}\left(1+\frac{2}{\left(1+2(L \sqrt{\frac{K_2}{K_1}}-1)^{-1}\right) \exp \left(2 \sqrt{K_1 K_2}(T-t)\right)-1}\right), \quad t \in[0, T] .
\end{equation*}
Then $\bar{P}$ solves the ODE $d \bar{P}_t=-\bar{H}\left(\bar{P}_t\right) d t,\ t \in[0, T]$, with terminal condition $\bar{P}_T=L$. The comparison principle implies that $P_{t}^{L}\leq \bar{P}_{t}$. Further note that $\bar{P}$ is monotonically increasing in $L$ and converges to
\begin{equation*}
\sqrt{\frac{K_1}{K_2}}\left(1+\frac{2}{\exp \left(2 \sqrt{K_1 K_2}(T-t)\right)-1}\right)
\end{equation*}
for all $t \in[0, T]$ as $L \rightarrow \infty$. This implies for all $t \in[0, T]$ that
\begin{equation*}
P_{t} ^{L}\leq \sqrt{\frac{K_1}{K_2}}+\frac{1}{K_2(T-t)},
\end{equation*}
which is controlled from above by $\kappa_{t}$ with $C_2:=\sqrt{\frac{K_1}{K_2}}\vee \frac{1}{K_{2}}$. Thus we
have proven $P_{t}^{L} \leq \kappa_{t}$ and, therefore, the upper bound in \eqref{u1}.

\textbf{Step 2:} Consider the bound of $P_{t}^{L}+\Phi_{\nu}^{L}(t,\nu)$. Observe that $P_{t}^{L}\nu+\Phi^{L}(t,\nu)$ is the decoupling field of the following FBSDE
\begin{equation}\label{conditional FBSDE}
\left\{\begin{aligned}
d \mathbb{E}[X_{t}^{L}|\mathcal{F}_{t}^{W^{0}}] & =  {\left[A_t \mathbb{E}[X_{t}^{L}|\mathcal{F}_{t}^{W^{0}}]-B_t^2 R_t^{-1} \mathbb{E}[Y_{t}^{L}|\mathcal{F}_{t}^{W^{0}}]-B_t h\left(t, \rho(t,-R_t^{-1} B_t \mathbb{E}[Y_{t}^{L} | \mathcal{F}_t^{W^{0}})\right)\right.} \\
& \left.\quad +f(t, \mathbb{E}[X_{t}^{L}| \mathcal{F}_t^{W^{0}}])+b\left(t, \rho(t,-R_t^{-1} B_t \mathbb{E}[Y_{t}^{L} | \mathcal{F}_t^{W^{0}}])\right)\right] d t, \\
d\mathbb{E}[Y_{t}^{L}|\mathcal{F}_{t}^{W^{0}}] &=  -\left[A_t\mathbb{E}[Y_{t}^{L}|\mathcal{F}_{t}^{W^{0}}]+Q_t\mathbb{E}[X_{t}^{L}|\mathcal{F}_{t}^{W^{0}}]+Q_t l(t, \mathbb{E}[X_{t}^{L} |\mathcal{F}_t^{W^{0}}])\right] d t+Z_t^{0} dW_t^0, \\
\mathbb{E}[X_0] &= \mathbb{E}[\xi], \quad \mathbb{E}[Y_{t}^{L}|\mathcal{F}^{W^{0}}_{T}]:= L\mathbb{E}[X_{t}^{L}|\mathcal{F}^{W^{0}}_{T}],
\end{aligned}\right.
\end{equation}
which is deduced by taking conditional expectation with respect to $\mathcal{F}^{W^{0}}$ to conditional mean field FBSDE \eqref{conditional mfFBSDE}.
Denote 
$$
P_{t}^{L}+\Phi_{\nu}^{L}(t,\nu) := \Psi^{L}_{t}.
$$
Note that \eqref{conditional FBSDE} is a classical FBSDE, with similar argument in \cite{hua2022unified}, we can get the dynamics of $\Psi^{L}$ as follows:
\begin{equation*}
\begin{aligned}
    d \Psi_{t}^{L} &= \Bigg[\bigg(B^{2}_{t}R_{t}^{-1}+\frac{R^{-1}_{t}B_{t}\left(b^{\prime}(t,\rho(t,-R_t^{-1} B_t \mathbb{E}[Y_{t}^{L}| \mathcal{F}_t^{W^{0}}])))-B_{t}h^{\prime}(t,\rho(t,-R_t^{-1} B_t \mathbb{E}[Y_{t}^{L} | \mathcal{F}_t^{W^{0}}])\right)}{1+h^{\prime}(t,\rho(t,-R_t^{-1} B_t \mathbb{E}[Y_{t}^{L} | \mathcal{F}_t^{W^{0}}]))}\bigg) (\Psi^{L}_{t})^{2}\\
    &\quad -2A_{t}\Psi^{L}_{t}-\left(Q_{t}+Q_{t}l^{\prime}(t,\mathbb{E}[X_{t}^{L}|\mathcal{F}_t^{W^{0}}])\right)\bigg]dt+Z_{t}^{0}dW_{t}^{0}.
\end{aligned}
\end{equation*}
 We define a generator $\mathcal{H}:[0,T]\times \mathbb{R} \rightarrow \mathbb{R}$ via
\begin{equation*}
\begin{aligned}
    \mathcal{H}(t,\psi) &:= -\left(B^{2}_{t}R_{t}^{-1}+\frac{R^{-1}_{t}B_{t}\left(b^{\prime}(t,\rho(t,-R_t^{-1} B_t \mathbb{E}[Y_{t}^{L}| \mathcal{F}_t^{W^{0}}])))-B_{t}h^{\prime}(t,\rho(t,-R_t^{-1} B_t \mathbb{E}[Y_{t}^{L} | \mathcal{F}_t^{W^{0}}])\right)}{1+h^{\prime}(t,\rho(t,-R_t^{-1} B_t \mathbb{E}[Y_{t}^{L} | \mathcal{F}_t^{W^{0}}]))}\right)
 \psi^{2}\\
    &\quad +2A_{t}\psi+\left(Q_{t}+Q_{t}l^{\prime}(t,\mathbb{E}[X_{t}^{L} |\mathcal{F}_t^{W^{0}}])\right)
\end{aligned}
\end{equation*}
such that $d\Psi_{t}^{L} = - \mathcal{H}(t,\Psi_{t}^{L})dt, t\in[0,T]$. 
To obtain the bound of $\Psi^{L}_{t}$, we define the following two generators $\hat{\mathcal{H}}$ and $\bar{\mathcal{H}}$: $\mathbb{R}\rightarrow \mathbb{R}$ via:
\begin{equation*}
\begin{aligned}
    \hat{\mathcal{H}}(\psi) &: = -(\frac{K^{2}}{\delta}+\frac{K^{2}+K^{3}}{\delta\varepsilon_{0}}) \psi^{2} -2K \psi,\\
    \bar{\mathcal{H}}(\psi) &: = -\frac{\delta^{2}}{2K}\psi^{2}+\frac{2K^{3}}{\delta^{2}}+K+K^{2}.
\end{aligned}
\end{equation*}
such that $d\hat{\Psi}_{t} = -\hat{\mathcal{H}}(\hat{\Psi}_{t})dt$ and $d\bar{\Psi}_{t} = -\bar{\mathcal{H}}(\bar{\Psi}_{t})dt$.\\
Under assumption (H), we have 
\begin{equation*}
\mathcal{G}_{t}:=B^{2}_{t}R_{t}^{-1}+\frac{R^{-1}_{t}B_{t}(b^{\prime}(t,\rho(t,-R_t^{-1} B_t \mathbb{E}[Y_{t}^{L}| \mathcal{F}_t^{W^{0}}])))-B_{t}h^{\prime}(t,\rho(t,-R_t^{-1} B_t \mathbb{E}[Y_{t}^{L} | \mathcal{F}_t^{W^{0}}]))}{1+h^{\prime}(t,\rho(t,-R_t^{-1} B_t \mathbb{E}[Y_{t}^{L} | \mathcal{F}_t^{W^{0}}]))}\geq \frac{\delta^{2}}{K},
\end{equation*}
and
\begin{equation*}
 \mathcal{G}_{t} \leq \frac{K^{2}}{\delta}+\frac{K^{2}+K^{3}}{\delta\varepsilon_{0}}.   
\end{equation*}
Using the following inequality
\begin{equation*}
    2A_{t}\psi =2\left( \sqrt{\frac{\mathcal{G}_{t}}{2}}\psi\right)\left(\sqrt{\frac{2}{\mathcal{G}}}A_{t}\right)\leq \frac{\mathcal{G}_{t}}{2}\psi^{2}+\frac{2}{\mathcal{G}_{t}}A_{t}^{2},
\end{equation*}
we obtain
\begin{equation*}
\begin{aligned}
    \mathcal{H}(t,\psi) &= -\mathcal{G}_{t}\psi^{2}+2A_{t}\psi+Q_{t}+Q_{t}l^{\prime}(t,\mathbb{E}[X_{t}^{L} |\mathcal{F}_t^{W^{0}}])  \\
    &\leq -\frac{1}{2}\mathcal{G}_{t}\psi^{2}+\frac{2}{\mathcal{G}_{t}}A_{t}^{2}+Q_{t}+Q_{t}l^{\prime}(t,\mathbb{E}[X_{t}^{L} |\mathcal{F}_t^{W^{0}}])\\
    &\leq -\frac{\delta^{2}}{2K}\psi^{2}+\frac{2K^{3}}{\delta^{2}}+K+K^{2}\\
    &=\bar{\mathcal{H}}(\psi),
    \end{aligned}  
\end{equation*}
and for $\psi \geq 0$, we have
\begin{equation*}
    \begin{aligned}
       \mathcal{H}(t,\psi) &= - \mathcal{G}_{t}\psi^{2}+2A_{t}\psi+Q_{t}+Q_{t}l^{\prime}(t,\mathbb{E}[X_{t}^{L} |\mathcal{F}_t^{W^{0}}]) \\
       &\geq -\left(\frac{K^{2}}{\delta}+\frac{K^{2}+K^{3}}{\delta\varepsilon_{0}}\right)\psi^{2}-2K\psi\\
       &=\hat{\mathcal{H}}(\psi)
    \end{aligned}
\end{equation*}
Let $K_{3}=\frac{K^{2}}{\delta}+\frac{K^{2}+K^{3}}{\delta\varepsilon_{0}} $, $K_{4}=\frac{2K^{3}}{\delta} +K + K^{2}$ and $K_{5} = \frac{\delta^{2}}{2K}$, define the deterministic, bounded and non-negative process $\hat{\Psi},\bar{\Psi}: [0,T] \rightarrow [0,\infty)$ via
\begin{equation*}
\begin{aligned}
  \hat{\Psi}_{t} &: =  \left(\frac{e^{2K(T-t)}}{L}+K_{3}(T-t) \frac{e^{2K(T-t)}-1}{2K(T-t)}\right)^{-1}, \quad t \in[0, T], \\
  \bar{\Psi}_{t} &:= \sqrt{\frac{K_4}{K_5}}\left(1+\frac{2}{\left(1+2(L \sqrt{\frac{K_5}{K_4}}-1)^{-1}\right) \exp \left(2 \sqrt{K_{4} K_5}(T-t)\right)-1}\right), \quad t \in[0, T] . 
\end{aligned}
\end{equation*}
It can be easily verify that  $\hat{\Psi}_{t}, \bar{\Psi}_{t}$ solves  $d\hat{\Psi}_{t} = -\hat{\mathcal{H}}(\hat{\Psi}_{t})dt$, with terminal condition $\hat{\Psi}_{T} = L$ and $d\bar{\Psi}_{t} = -\bar{\mathcal{H}}(\bar{\Psi}_{t})dt$, with terminal condition $\bar{\Psi}_{T} = L$,  respectively. Observe that $\hat{\Psi}_{t}<\hat{P}_{t}$, with similar argument in step 1, we can find appropriate constants $C_{3}$ and $C_{4}$ satisfying $C_{3}<C_{1}$ such that \eqref{u1u2} holds.

\textbf{Step 3:}
Under assumption (H)(iii), it is easy to verify  that $u^{L}(\cdot,0,0)=0
$, this implies directly that for $x>0$ and $\nu>0$,
\begin{equation*}
    \begin{aligned}
        \lambda_{t}^{L}x \leq  &P^{L}_{t}x \leq \kappa_{t}x,\\
        \tilde{\lambda}_{t}^{L}\nu \leq &P^{L}_{t}\nu+\Phi^{L}(t,\nu) \leq \tilde{\kappa}_{t}\nu.
    \end{aligned}
\end{equation*}
Therefore, for all $(t, x,\nu) \in[0, T) \times(0, \infty)\times (0, \infty)$, $u^L(t, x,\nu)$ is uniformly bounded independently of $L$.
\end{proof}

\begin{lemma}\label{Lnondecreasing}
Suppose assumption (H) holds. Then the mapping $(t,L,x,v)\rightarrow u^{L}(t,x,v) $ is continuous and non-decreasing in $L$.
\end{lemma}
\begin{proof}
 First, it is easy to verify $P_{t}^{L}$ is non-decreasing in $L$ by comparison theorem.
 
 Next, we introduce the following conditional mean field FBSDE:
\begin{equation}\label{conditional mfFBSDE with L}
\left\{\begin{aligned}
d X_t&=  {\left[A_t X_t-B_t^2 R_t^{-1} Y_t-B_t h\left(t, \rho(t,-R_t^{-1} B_t \mathbb{E}[Y_t |\mathcal{F}_t^{W^{0}}])\right)\right.} \\
& \left.\quad +f\left(t, \mathbb{E}[X_t |\mathcal{F}_t^{W^{0}}]\right)+b\left(t, \rho(t,-R_t^{-1} B_t \mathbb{E}[Y_t|\mathcal{F}_t^{W^{0}}])\right)\right] d t, \\
d L_{t} &= 0,\\
d Y_t&=  -\left[A_t Y_t+Q_tX_t+Q_t l\left(t, \mathbb{E}[X_t |\mathcal{F}_t^{W^{0}}]\right)\right] d t+Z_t d W_t+Z_t^{0} d W_t^0, \\
X_0 &= \xi,\quad L_{0} = L, \quad Y_T=(0 \vee X_{T} \wedge C) (0 \vee L_{T} \wedge C),
\end{aligned}\right.
\end{equation}
where $C>0$ is an arbitrary but fixed positive constant. We observe that if we choose the initial value $\xi$, $L$ between 0 and $C$, then the terminal
condition $Y_{T} = LX_{T}$ is satisfied and for any $\xi$ and $L$, we can always find the constant $C$ as $C>0$ can be chosen arbitrary large, which means that the unique solution $(X^{L}, Y^{L},Z^{L},Z^{0,L})$ of conditional mean field FBSDE \eqref{conditional mfFBSDE} solves FBSDE \eqref{conditional mfFBSDE with L}. We aim to obtain the uniqueness of the solution of FBSDE \eqref{conditional mfFBSDE with L}. To this end, we first consider in a small duration time. It follows from \cite{MFGbook2} that for a small enough time interval $[t_{0},T]$, the uniqueness of the solution of \eqref{conditional mfFBSDE with L} can be obtained. Let us denote $Y_{t} = u(t,X_{t},\mathbb{E}[X_{t}|\mathcal{F}_{t}^{W^{0}}],L_{t})$ as the decoupling field of FBSDE \eqref{conditional mfFBSDE with L} and $\nu_{t} = \mathbb{E}[X_{t}|\mathcal{F}^{W^{0}}_{t}]$. The uniqueness of the solution of FBSDE \eqref{conditional mfFBSDE with L} on small time interval implies that $(X^{L}, Y^{L},Z^{L},Z^{0,L})$ and $(X,Y,Z,Z^{0})$ coincide on time interval $[t_{0},T]$.\\
Next, we consider in the time interval $[t_{0},T]$. By \eqref{u_p_phi}, we have
\begin{equation}\label{equality}
Y_{t}=u(t,X_{t},\mathbb{E}[X_{t}|\mathcal{F}_{t}^{W^{0}}],L_{t}) =Y^{L}_{t}= u^{L}(t,X_{t}^{L},\mathbb{E}[X_{t}^{L}|\mathcal{F}_{t}^{W^{0}}])= P_{t}^{L}X^{L}_{t}+ \varphi_{t}^{L},\quad t\in[t_{0},T],
\end{equation}
where $u^{L}$ is the decoupling field of conditional mean field FBSDE \eqref{conditional mfFBSDE}. Therefore, we have
\begin{equation}\label{u-ul-1}
    \frac{\partial u}{ \partial L }+ \frac{\partial u}{\partial x}\frac{dX_{t}}{dL}+ \frac{\partial u}{\partial \nu} \frac{d\nu_{t}}{dL} 
    = \frac{dP_{t}^{L}}{dL}X_{t}^{L}+P_{t}^{L}\frac{dX_{t}^{L}}{dL}+\frac{d \varphi_{t}^{L}}{d L}.
\end{equation}
The equivalence of $(X_{t},Y_{t})$ and $(X^{L}_{t},Y_{t}^{L})$ indicates that
\begin{equation}\label{u-ul-2}
    \frac{\partial u}{\partial x} = P_{t}^{L},\quad \frac{\partial u}{\partial\nu} = \Phi^{L}_{\nu}(t,\nu),\quad \frac{d\nu}{dL} = \frac{d\nu^{L}}{dL}.
\end{equation}
For simplicity of notations, we denote 
\begin{equation}\label{notation}
    \frac{dP_{t}^{L}}{d L} = \hat{P}_{t}, \quad \frac{d \varphi_{t}^{L}}{dL} = \hat{\varphi}_{t}, \quad \frac{d\nu_{t}^{L}}{d L} = \hat{\nu}_{t},\quad \frac{\partial u}{\partial L} = \gamma_{t},\quad \Phi^{L}_{\nu}(t,\nu) = \Lambda_{t}.
\end{equation}
Combine \eqref{u-ul-1} and \eqref{u-ul-2}, we obtain
\begin{equation}\label{gammaform}
    \gamma_{t} = \hat{P}_{t}X_{t}+\hat{\varphi}_{t} - \Lambda_{t}\hat{\nu}_{t}.
\end{equation}
For simplicity of notations, the superscript $L$ for processes $P^{L},\varphi^{L},\nu^{L}$ will be omitted and we denote
\begin{equation*}
    \begin{aligned}
    \tilde{h}_{t}&:=h(t,\rho(t,-R^{-1}_{t}B_{t}(P_{t}\nu_{t}+\varphi_{t}))),\\
    \tilde{b}_{t}&:=b(t,\rho(t,-R^{-1}_{t}B_{t}(P_{t}\nu_{t}+\varphi_{t}))),\\
    \tilde{h}_{t}^{\prime}&:= h^{\prime}(t,\rho(t,-R^{-1}_{t}B_{t}(P_{t}\nu_{t}+\varphi_{t})))\frac{-R_{t}^{-1}B_{t}}{1+h^{\prime}(t,-R^{-1}_{t}B_{t}(P_{t}\nu_{t}+\varphi_{t}))},\\
        \tilde{b}_{t}^{\prime} & := b^{\prime}(t,\rho(t,-R^{-1}_{t}B_{t}(P_{t}\nu_{t}+\varphi_{t})))\frac{-R_{t}^{-1}B_{t}}{1+h^{\prime}(t,-R^{-1}_{t}B_{t}(P_{t}\nu_{t}+\varphi_{t}))}.
    \end{aligned}
\end{equation*}
From \eqref{p equation} and \eqref{expectation FBSDE} and using \eqref{equality}, it is easy to verify  that $\hat{P},\ \hat{\varphi},\ \hat{\nu},\,\ X$ satisfy the following dynamics
\begin{align*}
        d\hat{P}_{t} &= (2B_{t}^{2}R_{t}^{-1}P_{t}\hat{P}_{t}-2A_{t}\hat{P}_{t})dt,\\
        d\hat{\varphi}_{t} &= -\bigg[A_{t}\hat{\varphi}_{t}-B^{2}_{t}R_{t}^{-1}P_{t}\hat{\varphi}_{t}-B^{2}_{t}R_{t}^{-1}\hat{P}_{t}\varphi_{t}+\hat{P}_{t}f(t,\nu_{t})+P_{t}f^{\prime}(t,\nu_{t})\hat{\nu}_{t}+\hat{P}_{t}\tilde{b}_{t}\\
         &\quad +P_{t}\tilde{b}_{t}^{\prime}(\hat{P}_{t}\nu_{t}+P_{t}\hat{\nu}_{t})+P_{t}\tilde{b}^{\prime}_{t}\hat{\varphi}_{t}+Q_{t}l^{\prime}(t,\nu_{t})\hat{\nu}_{t}-\hat{P}_{t}B_{t}\tilde{h}_{t}-P_{t}B_{t}\tilde{h}_{t}^{\prime}(\hat{P}_{t}\nu_{t}+P_{t}\hat{\nu}_{t})-P_{t}B_{t}\tilde{h}_{t}^{\prime}\hat{\varphi}_{t}\bigg]dt\\
    &\quad +\frac{d\Gamma_{t}}{dL}dW_{t},\\
     d\hat{\nu}_{t} &= \bigg[A_{t}\hat{\nu_{t}}-B^{2}_{t}R^{-1}_{t}P_{t}\hat{\nu}_{t}-B^{2}_{t}R^{-1}_{t}\hat{P}_{t}\nu_{t}-B^{2}_{t}R^{-1}_{t}\hat{\varphi}_{t}-B_{t}\tilde{h}_{t}^{\prime}(\hat{P}_{t}\nu_{t}+P_{t}\hat{\nu}_{t})\\
        &\quad -B_{t}\tilde{h}^{\prime}_{t}\hat{\varphi}_{t}+f^{\prime}(t,\nu_{t})\hat{\nu}_{t}+\tilde{b}_{t}^{\prime}(\hat{P}_{t}\nu_{t}+P_{t}\hat{\nu}_{t})+\tilde{b}_{t}^{\prime}\hat{\varphi}_{t}\bigg]dt,\\
       dX_{t} &= \bigg[(A_{t}-B_{t}^{2}R_{t}^{-1}P_{t})X_{t}-B^{2}_{t}R^{-1}_{t}\varphi_{t}-B_{t}\tilde{h}_{t}+f(t,\nu_{t})+\tilde{b}_{t}\bigg]dt,
\end{align*}
and note that \eqref{expectation FBSDE} is a classical FBSDE, with similar argument in \cite{hua2022unified}, we can get the dynamics of $\Lambda_{t}$ as follows:
\begin{equation*}
    \begin{aligned}
         d\Lambda_{t} &= \bigg[\left(B^{2}_{t}R_{t}^{-1}+B_{t}\tilde{h}_{t}^{\prime}-\tilde{b}_{t}^{\prime}\right)\Lambda_{t}^{2}-\left(2(A_{t}-B^{2}_{t}R_{t}^{-1}P_{t})-B_{t}P_{t}\tilde{h}_{t}^{\prime}+f^{\prime}(t,\nu_{t})+P_{t}\tilde{b}_{t}+P_{t}\tilde{b}_{t}^{\prime}-P_{t}B_{t}\tilde{h}^{\prime}_{t}\right)\Lambda_{t}\\
    &\quad -\left(P_{t}f^{\prime}(t,\nu_{t})+Q_{t}l^{\prime}(t,\nu_{t})+P_{t}^{2}\tilde{b}_{t}^{\prime}-P_{t}^{2}B_{t}\tilde{h}_{t}^{\prime}\right)\bigg]dt +\mathcal{Z}_{t}dW_{t}^{0},
    \end{aligned}
\end{equation*}
The dynamics of $\gamma$ are now deduced from those of $\hat{\nu}$, $\hat{\varphi}$, $\hat{P},\Lambda,X$ using the product rule. For all $t \in [t_{0},T]$, it holds 
\begin{equation}\label{eq-1}
\begin{aligned}
    d \gamma_{t} &= X_{t}d \hat{P}_{t} + \hat{P}_{t}dX_{t} + d\hat{\varphi}_{t}- \hat{\nu}_{t}d\Lambda_{t} -\Lambda_{t} d \hat{\nu}_{t}\\
 &= \bigg\{\Big[-P_{t}\tilde{b}_{t}^{\prime} + B_{t}^{2}R^{-1}_{t}P_{t}+P_{t}B_{t}\tilde{h}_{t}^{\prime}+B_{t}\Lambda_{t}\tilde{h}^{\prime}_{t}-\Lambda_{t}\tilde{b}_{t}^{\prime}-A_{t}+B_{t}R^{-1}_{t}\Lambda_{t}\Big]\hat{\varphi}_{t}\\ 
    &\quad + \Big[-P_{t}\tilde{b}_{t}^{\prime} + B_{t}^{2}R^{-1}_{t}P_{t}+P_{t}B_{t}\tilde{h}_{t}^{\prime}+B_{t}\Lambda_{t}\tilde{h}_{t}^{\prime}-\Lambda_{t} \tilde{b}^{\prime}_{t}-A_{t}+B_{t}R^{-1}_{t}\Lambda_{t}\Big] (-\Lambda_{t} \hat{\nu}_{t})\\
    & \quad +\Big[B^{2}_{t}R_{t}^{-1}P_{t}-A_{t}\Big] \hat{P}_{t}X_{t}+ \Big[-P_{t}\tilde{b}_{t}^{\prime} +P_{t}B_{t}\tilde{h}_{t}^{\prime}+B_{t}\Lambda_{t}\tilde{h}_{t}^{\prime}-\Lambda_{t} \tilde{b}_{t}^{\prime}+B_{t}R^{-1}_{t}\Lambda_{t}\Big]\hat{P}_{t}\nu_{t}dt\bigg\}+\hat{\nu}_{t}\mathcal{Z}_{t}dW_{t}^{0}\\
    & = \bigg\{\Big[B^{2}_{t}R^{-1}_{t}P_{t}-A_{t}\Big] (\hat{\varphi}_{t}+\hat{P}_{t}X_{t}-\Lambda_{t} \hat{\nu}_{t}) \\
    &\quad +\Big[-P_{t}\tilde{b}_{t}^{\prime} +P_{t}B_{t}\tilde{h}_{t}^{\prime}+B_{t}\Lambda_{t}\tilde{h}_{t}^{\prime}-\Lambda_{t}\tilde{b}_{t}^{\prime}+B_{t}R^{-1}_{t}\Lambda_{t}\Big] (\hat{\varphi}_{t}+\hat{P}_{t}\nu_{t}-\Lambda_{t} \hat{\nu}_{t})dt\bigg\}+\hat{\nu}_{t}\mathcal{Z}_{t}dW_{t}^{0}.
\end{aligned}
\end{equation}
From \eqref{gammaform}, and the fact that $\hat{P}_{t}$ are deterministic, $\hat{\nu}_{t},\hat{\varphi}_{t}, \Lambda_{t}$ are all $\mathbb{F}^{0}$-measurable and $\mathbb{E}[X_{t}|\mathcal{F}^{W^{0}}_{t}] = \nu_{t}$,  we obtain that
\begin{equation}\label{eq-2}
    \mathbb{E}[\gamma_{t}|\mathcal{F}_{t}^{W^{0}}] = \hat{P}_{t}\nu_{t}+\hat{\varphi}_{t} - \Lambda_{t}\hat{\nu}_{t}.
\end{equation}
Substituting \eqref{eq-2} into \eqref{eq-1}, we obtain
\begin{equation*}
     d\gamma_{t} = \bigg\{\left(B_{t}^{2}R_{t}^{-1}P_{t}-A_{t}\right)\gamma_{t}+\left(-P_{t}\tilde{b}^{\prime}_{t} +P_{t}B_{t}\tilde{h}_{t}^{\prime}+B_{t}\Lambda_{t}\tilde{h}_{t}^{\prime}-\Lambda_{t} \tilde{b}_{t}^{\prime}+B_{t}R^{-1}_{t}\Lambda_{t}
\right)\mathbb{E}[\gamma_{t}|\mathcal{F}_{t}^{W^{0}}]\bigg\}dt+\hat{\nu}_{t}\mathcal{Z}_{t}dW_{t}^{0}.
\end{equation*}
It is obvious that the process $\gamma$ satisfies conditional mean field BSDEs with linear driver and bounded coefficient as $P_{t}$ and $\Lambda_{t}$ are uniformly bounded according to lemma \ref{uniformly bounded}. Note that for all $\tilde{x}, \tilde{\nu},\tilde{L} \in \mathbb{R}$ it holds
$$
u(T, \tilde{x},\tilde{\nu}, \tilde{L})=(0 \vee(\tilde{x}) \wedge C)(0 \vee \tilde{L} \wedge C),
$$
such that $\frac{u(T, x,\nu, L)}{\partial L}$ are both nonnegative and uniformly bounded (with the respective bounds depending on $C$). Therefore, using comparison theorem, we obtain that $\frac{\partial u}{\partial L}$ is nonnegative and uniformly bounded, that is, independently of $t_0$. Moreover, $\frac{\partial u}{\partial x}=P_{t}^{L}$ and $\frac{\partial u}{\partial \nu}=\Phi^{L}_{\nu}$ are also uniformly bounded. Then, following similar arguments as \cite{hua2022unified}, we can extend the solution from small duration to the whole interval and get the existence and uniqueness of solutions to the conditional mean field FBSDE \eqref{conditional mfFBSDE with L} on the whole time interval. Therefore, the unique solutions of FBSDE \eqref{conditional mfFBSDE with L} and FBSDE \eqref{conditional mfFBSDE} coincide on $[0,T]$. Therefore, the monotonicity of $u$ in $L$ is inherited by $u^{L}$. 
\end{proof}

Lemma \ref{uniformly bounded} states that for all $(t, x,\nu) \in[0, T) \times(0, \infty)\times (0, \infty)$, $u^L(t, x,\nu)$ is uniformly bounded independently of $L$ . Moreover, lemma \ref{Lnondecreasing} gives that $u^L(t, x,\nu)$ is non-decreasing with respect of $L$. These yield pointwise convergence of $u^{L}(t,x,L)$ for $L\rightarrow \infty$. Now we can define
\begin{equation}\label{u infty defination}
u^{\infty}(s,x,\nu) := \lim_{L \rightarrow \infty} u^{L}(s,x,\nu),\quad s\in[0,T),~x\in(0, \infty),~\nu\in(0, \infty).
\end{equation}
Moreover, since the sequence $P^{L}$ is uniformly bounded and non-decreasing in $L$, we can also define
\begin{equation}
P^{\infty}_{t}:= \lim_{L \rightarrow \infty}P_{t}^{L}.   
\end{equation}

Now we would like to obtain the convergence of $X^{L}$ for $L\rightarrow \infty$. It is worth emphasizing that we can not  get the non-increasing property of $X^{L}$ with respect to $L$ although we can obtain the non-increasing property of $\mathbb{E}[X_{t}^{L}|\mathcal{F}_{t}^{W^{0}}]$, which is another different point compared with the case without mean field term in \cite{2020Fromm}. 
\begin{lemma}\label{Xinfty}
  Under assumption (H), the sequence $X_{t}^{L}$ converges in $L^2_{\mathbb{F}}\left([0, T); \mathbb{R}\right)$ for any $t \in\left[0, T\right)$ to some limit $X^{\infty}_{t}$. 
\end{lemma}
\begin{proof}
 First, we can obtain that $\mathbb{E}[X_{t}^{L}|\mathcal{F}_{t}^{W^{0}}]$ is non-increasing in $L$. In fact, since
\begin{equation*}
\begin{aligned}
d \mathbb{E}[X_t^{L} |\mathcal{F}_{t}^{W^{0}}]
&=  {\left[A_t \mathbb{E}[X_t^{L} |\mathcal{F}_{t}^{W^{0}}]-B_t^2 R_t^{-1} u^{L}(t, \mathbb{E}[X_t^{L} |\mathcal{F}_{t}^{W^{0}}],\mathbb{E}[X^{L}_{t}|\mathcal{F}_{t}^{W^{0}}])+f(t, \mathbb{E}[X_t^{L} |\mathcal{F}_{t}^{W^{0}}])\right.} \\
&\quad  -B_t h\left(t, \rho(t,-R_t^{-1} B_t u^{L}(t,\mathbb{E}[X^{L}_{t}|\mathcal{F}_{t}^{W^{0}}],\mathbb{E}[X^{L}_{t}|\mathcal{F}_{t}^{W^{0}}]))\right)\\
&\quad\left.+b\left(t, \rho(t,-R_t^{-1} B_t  u^{L}(t,\mathbb{E}[X^{L}_{t}|\mathcal{F}_{t}^{W^{0}}],\mathbb{E}[X^{L}_{t}|\mathcal{F}_{t}^{W^{0}}]))\right)\right] d t\\
& :=F_{t}dt,
\end{aligned}
\end{equation*}
we have
\begin{equation*}
F_{L} = -\frac{B_{t}R_{t}^{-1}(B_{t}+b^{\prime}_{t}(t,\cdot))}{1+h^{\prime}(t,\cdot)}u_{L}\leq 0
\end{equation*}
The comparison principle implies that the sequence $\mathbb{E}[X_{t}^{L}|\mathcal{F}_{t}^{W^{0}}]$ is non-increasing in $L$.
Combined with the non-negativity of $X^{L}$, we obtain the sequence $\mathbb{E}[X_{t}^{L}|\mathcal{F}_{t}^{W^{0}}]$ converges to some limit $\mathcal{M}_{t}$ for $L\rightarrow \infty$. Moreover, from lemma \ref{x trajectory},  it holds for all $t\in[0,T)$ and $L>0$, $\max(\mathbb{E}[X_{t}^{L}|\mathcal{F}_{t}^{W^{0}}],\mathcal{M}_{t})\leq \mathbb{E}[\xi]$. By the dominated convergence theorem, the sequence $\mathbb{E}[X_{t}^{L}|\mathcal{F}_{t}^{W^{0}}]$ converges in the space $L^{2}_{\mathbb{F}}([0,T);\mathbb{R})$.
Next, we would like to prove the sequence $X_{t}^{L}$ also converges in $L^{2}_{\mathbb{F}}([0,T);\mathbb{R})$. \\
For any $L_{1},L_{2}>0$, applying It\^{o}'s formula to $|X^{L_{1}}_{t}-X^{L_{2}}_{t}|^{2}$, we obtain,
\begin{align*}
 |X^{L_{1}}_{t}-X^{L_{2}}_{t}|^{2}
&= \int_{0}^{t} 2(X^{L_{1}}_{s}-X^{L_{2}}_{s})\Bigg\{A_t (X^{L_{1}}_{s}-X^{L_{2}}_{s})+f(s, \mathbb{E}[X_s^{L_{1}} | \mathcal{F}_s^{W^{0}}])-f(s, \mathbb{E}[X_s^{L_{2}} | \mathcal{F}_s^{W^{0}}])\\
&\quad-B_s^2 R_s^{-1} \left(u^{L_{1}}(s, X^{L_{1}}_{s},\mathbb{E}[X^{L_{1}}_{s}|\mathcal{F}_{s}^{W^{0}}])-u^{L_{2}}(s, X^{L_{2}}_{s},\mathbb{E}[X^{L_{2}}_{s}|\mathcal{F}_{s}^{W^{0}}])\right) \\
&\quad  -B_s\bigg(h(t, \rho(t,-R_s^{-1} B_s u^{L_{1}}(s,\mathbb{E}[X^{L_{1}}_{s}|\mathcal{F}_{s}^{W^{0}}],\mathbb{E}[X^{L_{1}}_{s}|\mathcal{F}_{s}^{W^{0}}])))\\
&\quad -h(s, \rho(s,-R_s^{-1} B_s u^{L_{2}}(s,\mathbb{E}[X^{L_{2}}_{s}|\mathcal{F}_{s}^{W^{0}}],\mathbb{E}[X^{L_{2}}_{s}|\mathcal{F}_{s}^{W^{0}}])))\bigg)\\
&\quad+b\left(s, \rho(t,-R_s^{-1} B_s  u^{L_{1}}(s,\mathbb{E}[X^{L_{1}}_{s}|\mathcal{F}_{s}^{W^{0}}],\mathbb{E}[X^{L_{1}}_{s}|\mathcal{F}_{s}^{W^{0}}]))\right)\\
&\quad -b\left(s, \rho(t,-R_s^{-1} B_s  u^{L_{2}}(s,\mathbb{E}[X^{L_{2}}_{s}|\mathcal{F}_{s}^{W^{0}}],\mathbb{E}[X^{2}_{s}|\mathcal{F}_{s}^{W^{0}}]))\right)\Bigg\} d s.   
\end{align*}
Using \eqref{u_p_phi}, we have
 \begin{equation*}
\begin{aligned}
        &\quad u^{L_{1}}(s, X^{L_{1}}_{s},\mathbb{E}[X^{L_{1}}_{s}|\mathcal{F}_{s}^{W^{0}}])-u^{L_{2}}(s, X^{L_{2}}_{s},\mathbb{E}[X^{L_{2}}_{s}|\mathcal{F}_{s}^{W^{0}}])\\
        & = u^{L_{1}}(s, X^{L_{1}}_{s},\mathbb{E}[X^{L_{1}}_{s}|\mathcal{F}_{s}^{W^{0}}])-u^{L_{1}}(s, X^{L_{2}}_{s},\mathbb{E}[X^{L_{2}}_{s}|\mathcal{F}_{s}^{W^{0}}])\\
        &\quad \quad +u^{L_{1}}(s, X^{L_{2}}_{s},\mathbb{E}[X^{L_{2}}_{s}|\mathcal{F}_{s}^{W^{0}}])-u^{L_{2}}(s, X^{L_{2}}_{s},\mathbb{E}[X^{L_{2}}_{s}|\mathcal{F}_{s}^{W^{0}}])\\
       &=P_{s}^{L_{1}}(X^{L_{1}}_{s}-X^{L_{2}}_{s})+\Phi^{L_{1}}(s,\mathbb{E}[X^{L_{1}}_{s}|\mathcal{F}_{s}^{W^{0}}])-\Phi^{L_{1}}(s,\mathbb{E}[X^{L_{2}}_{s}|\mathcal{F}_{s}^{W^{0}}])\\
        &\quad \quad +u^{L_{1}}(s, X^{L_{2}}_{s},\mathbb{E}[X^{L_{2}}_{s}|\mathcal{F}_{s}^{W^{0}}])-u^{L_{2}}(s, X^{L_{2}}_{s},\mathbb{E}[X^{L_{2}}_{s}|\mathcal{F}_{s}^{W^{0}}]).
       \end{aligned}
\end{equation*}
Then based on the Lipschitz continuity property of coefficients, we have for any $\varepsilon>0$,
\begin{equation*}
    \begin{aligned}
         &|X^{L_{1}}_{t}-X^{L_{2}}_{t}|^{2}+\int_{0}^{t}(2B_{s}^{2}R_{s}^{-1}P^{L_{1}}_{s}-2A_{s})|X_{s}^{L_{1}}-X_{s}^{L_{2}}|^{2}ds\\
         &\leq \int_{0}^{t}\bigg[\varepsilon|X_{s}^{L_{1}}-X_{s}^{L_{2}}|^{2}+\frac{C}{\varepsilon}\bigg(\left|\Phi^{L_{1}}(t,\mathbb{E}[X^{L_{1}}_{s}|\mathcal{F}_{s}^{W^{0}}])-\Phi^{L_{1}}(s,\mathbb{E}[X^{L_{2}}_{s}|\mathcal{F}_{s}^{W^{0}}])\right|^{2}\\
         &\quad\quad  +\left|u^{L_{1}}(t, \mathbb{E}[X^{L_{1}}_{s}|\mathcal{F}_{s}^{W^{0}}],\mathbb{E}[X^{L_{1}}_{s}|\mathcal{F}_{s}^{W^{0}}])-u^{L_{1}}(s, \mathbb{E}[X^{L_{2}}_{s}|\mathcal{F}_{s}^{W^{0}}],\mathbb{E}[X^{L_{2}}_{s}|\mathcal{F}_{s}^{W^{0}}])\right|^{2}\\
         &\quad \quad+\left|u^{L_{1}}(s, \mathbb{E}[X^{L_{2}}_{s}|\mathcal{F}_{s}^{W^{0}}],\mathbb{E}[X^{L_{2}}_{s}|\mathcal{F}_{s}^{W^{0}}])-u^{L_{2}}(s, \mathbb{E}[X^{L_{2}}_{s}|\mathcal{F}_{s}^{W^{0}}],\mathbb{E}[X^{L_{2}}_{s}|\mathcal{F}_{s}^{W^{0}}])\right|^{2}\bigg)\bigg]ds.
    \end{aligned}
\end{equation*}
Based on the uniformly Lipschitz continuity of $u^{L}_{t}$ and $\Phi^{L}_{t}$ on $[0,T)$ obtained in lemma \ref{uniformly bounded} and the assumption $A_{t}\leq 0$, we have for small enough $\varepsilon>0$, 
\begin{equation}\label{xcauchy}
    \begin{aligned}
&\quad\mathbb{E}\int_{0}^{t}|X_{s}^{L_{1}}-X_{s}^{L_{2}}|^{2}ds\\
       &\leq \bar{C}\mathbb{E}\int_{0}^{t}\bigg[\left|u^{L_{1}}(s, \mathbb{E}[X^{L_{2}}_{s}|\mathcal{F}_{s}^{W^{0}}],\mathbb{E}[X^{L_{2}}_{s}|\mathcal{F}_{s}^{W^{0}}])-u^{L_{2}}(t, \mathbb{E}[X^{L_{2}}_{s}|\mathcal{F}_{s}^{W^{0}}],\mathbb{E}[X^{L_{2}}_{s}|\mathcal{F}_{s}^{W^{0}}])\right|^{2}\\
       &\quad+\left|\mathbb{E}[X^{L_{1}}_{s}|\mathcal{F}_{s}^{W^{0}}]-\mathbb{E}[X^{L_{2}}_{s}|\mathcal{F}_{s}^{W^{0}}]\right|^{2}\bigg]ds,
    \end{aligned}
\end{equation}
for some constant $\bar{C}$ independent of $L^{1}$ and $L^{2}$.

Lemma \ref{uniformly bounded} indicates that the sequence $u^{L}(t,\cdot,\cdot)$ is uniformly equicontinuous when $t\in[0,T)$ and uniformly bounded, and lemma \ref{x trajectory} states that the sequence $u^{L}$ is defined on a closed and bounded interval. Therefore, it follows from Arzel\`{a}–Ascoli theorem that there exists a subsequence $\{u^{L_{K}}\}_{K\in \mathbb{N}}$ that converges uniformly. Recalling the non-decreasing property of sequence of $u^{L}$ proved by lemma \ref{Lnondecreasing}, we can further obtain the sequence $u^{L}$  converges uniformly. Then, by the dominated convergence theorem, the sequence $u^{L}$ converges in the space $L^{2}_{\mathbb{F}}([0,T);\mathbb{R})$. Combined with \eqref{xcauchy}, we obtain $X_{t}^{L}$ is a fundamental sequence in $L^{2}_{\mathbb{F}}([0,T);\mathbb{R})$, and thus converges to some limit,
\begin{equation}\label{xinfty}
    X^{\infty}_{t}:= \lim_{L\rightarrow \infty} X^{L}_{t}, \quad  t\in[0,T).
\end{equation}  
 As $X^{L}$ is nonnegative for all $L>0$ by lemma \ref{x trajectory}, $X^{\infty}$ is nonnegative.
\end{proof}
\begin{lemma}\label{xT=0}
 Under assumption (H),
 it holds that for all $t\in[0,T)$,
 $$
 \lim_{L \rightarrow \infty} u^L\left(t,X^{L}_{t},\mathbb{E}[X^{L}_{t}|\mathcal{F}_{t}^{W^{0}}]\right) \rightarrow u^{\infty}\left(t,X_{t}^{\infty},\mathbb{E}[X^{\infty}_{t}|\mathcal{F}_{t}^{W^{0}}]\right).
 $$ 
 Moreover, $X^{\infty}$ is the unique solution of the following conditional mean field SDE
\begin{equation}\label{equ-x-infty}
    \begin{aligned}
        d X_t^{\infty}= & {\left[A_t X_t^{\infty}-B_t^2 R_t^{-1} u^{\infty}(t,X_{t}^{\infty},\mathbb{E}[X^{\infty}_{t}|\mathcal{F}_{t}^{W^{0}}])+f\left(t, \mathbb{E}\left[X_t^{\infty}| \mathcal{F}_t^0\right]\right)\right.}\\
        &-B_t h\left(t, \rho\left(t,-R_t^{-1} B_t u^{\infty}(t,\mathbb{E}[X^{\infty}_{t}|\mathcal{F}_{t}^{W^{0}}],\mathbb{E}[X^{\infty}_{t}|\mathcal{F}_{t}^{W^{0}}])\right)\right) \\
& \left.+b\left(t, \rho\left(t,-R_t^{-1} B_t  u^{\infty}(t,\mathbb{E}[X^{\infty}_{t}|\mathcal{F}_{t}^{W^{0}}],\mathbb{E}[X^{\infty}_{t}|\mathcal{F}_{t}^{W^{0}}])\right)\right)\right] dt,
\end{aligned}
\end{equation}
In particular, almost all paths of $X^{\infty}$ are absolutely continuous and non-increasing on $[0,T)$ and  $X_{T}^{\infty} = 0$.
\end{lemma}
\begin{proof}
First it holds that 
\begin{equation*}
\begin{aligned}
   &\quad \left|u^{L}(s,X^{L}_{s},\mathbb{E}[X^{L}_{s}|\mathcal{F}^{W^{0}}_{s}])-u^{\infty}(s,X^{\infty}_{s},\mathbb{E}[X^{\infty}_{s}|\mathcal{F}^{W^{0}}_{s}])\right|\\
   &\leq \left|u^{L}(s,X^{L}_{s},\mathbb{E}[X^{L}_{s}|\mathcal{F}^{W^{0}}_{s}])-u^{L}(s,X^{\infty}_{s},\mathbb{E}[X^{\infty}_{s}|\mathcal{F}^{W^{0}}_{s}])\right| \\
   & \quad \quad +\left|u^{L}(s,X^{\infty}_{s},\mathbb{E}[X^{\infty}_{s}|\mathcal{F}^{W^{0}}_{s}])-u^{\infty}(s,X^{\infty}_{s},\mathbb{E}[X^{\infty}_{s}|\mathcal{F}^{W^{0}}_{s}])\right|.
\end{aligned}   
\end{equation*}
By construction, both summands on the right-hand side converge to 0 as $L\rightarrow \infty$. Dominated convergence theorem implies for all $t\in [0,T)$, 
\begin{align*}
 & \quad \lim_{L \rightarrow \infty} \int_{0}^{t}\bigg[A_s X_s^{L}-B_s^2 R_s^{-1} u^{L}(s,X_{s}^{L},\mathbb{E}[X^{L}_{s}|\mathcal{F}_{s}^{W^{0}}])+f(s, \mathbb{E}[X_s^{L} | \mathcal{F}_s^{W^{0}}]) \\
   &\quad \quad \quad \quad -B_s h\left(s, \rho\left(s,-R_s^{-1} B_s u^{L}(s,\mathbb{E}[X^{L}_{s}|\mathcal{F}_{s}^{W^{0}}],\mathbb{E}[X^{L}_{s}|\mathcal{F}_{s}^{W^{0}}])\right)\right)\\
& \quad \quad \quad \quad +b\left(s, \rho\left(s,-R_s^{-1} B_su^{L}(s,\mathbb{E}[X^{L}_{s}|\mathcal{F}_{s}^{W^{0}}],\mathbb{E}[X^{L}_{s}|\mathcal{F}_{s}^{W^{0}}])\right)\right)\bigg] ds\\
& = \int_{0}^{t}\bigg[A_s X_s^{\infty}-B_s^2 R_s^{-1} u^{\infty}(s,X_{s}^{\infty},\mathbb{E}[X^{\infty}_{s}|\mathcal{F}_{s}^{W^{0}}])+f(s, \mathbb{E}[X_s^{\infty} | \mathcal{F}_s^{W^{0}}]) \\
   &\quad \quad \quad \quad -B_s h\left(s, \rho(s,-R_s^{-1} B_s u^{\infty}(s,\mathbb{E}[X^{\infty}_{s}|\mathcal{F}_{s}^{W^{0}}],\mathbb{E}[X^{\infty}_{s}|\mathcal{F}_{s}^{W^{0}}]))\right)\\
& \quad \quad \quad \quad +b\left(s, \rho(s,-R_s^{-1} B_s u^{\infty}(s,\mathbb{E}[X^{\infty}_{s}|\mathcal{F}_{s}^{W^{0}}],\mathbb{E}[X^{\infty}_{s}|\mathcal{F}_{s}^{W^{0}}]))\right)\bigg]ds,
\end{align*}
which proves that $X^{\infty}$ satisfies \eqref{equ-x-infty}. Lipschiz continuity of coefficients and of $u^{\infty}$ implies that $X^{\infty}$ is the unique solution of the conditional mean field SDE. \\
Moreover, under assumption (H), the following function 
\begin{equation*}
   \mathcal{G}(s,u):= -B_s h\left(s, \rho\left(s,-R_s^{-1} B_s u\right)\right)+b\left(s, \rho\left(s,-R_s^{-1} B_su\right)\right)
\end{equation*}
satisfies that $\mathcal{G}(s,0)=0$ and is decreasing in $u$. Recalling the nonnegatively of $X^{\infty}$, we can obtain the nonnegatively of the drift coefficient of $X^{\infty}$, and further the non-increasing of the path $X^{\infty}$.
Using \eqref{u1u2} and \eqref{u1}, we have
\begin{equation}\label{u_lower}
\begin{aligned}
   u^{\infty}(s,X^{\infty}_{s},\mathbb{E}[X_s^{\infty}| \mathcal{F}_s^{W^{0}}])& = \lim_{L\rightarrow\infty}(P^{L}_{s}X^{\infty}_{s}+\Phi^{L}(s,\mathbb{E}[X_s^{\infty}| \mathcal{F}_s^{W^{0}}]))  \\
   & = \lim_{L\rightarrow\infty}\left(P^{L}_{s}(X^{\infty}_{s}-\mathbb{E}[X_s^{\infty}| \mathcal{F}_s^{W^{0}}])+(P^{L}_{s}\mathbb{E}[X_s^{\infty}| \mathcal{F}_s^{W^{0}}]+\Phi^{L}(s,\mathbb{E}[X_s^{\infty}| \mathcal{F}_s^{W^{0}}]))\right)\\
   &>=\frac{C_{1}(X^{\infty}_{s}-\mathbb{E}[X_s^{\infty}| \mathcal{F}_s^{W^{0}}])}{T-t}+\frac{C_{3}\mathbb{E}[X_s^{\infty}| \mathcal{F}_s^{W^{0}}]}{T-t}\\
   &\geq \frac{C_{1}X_{s}^{\infty}}{T-t}.
   \end{aligned}
\end{equation}
Using \eqref{u_lower} and since $X^{\infty}$ has nonnegative and non-increasing paths and $u^{\infty}$ is nonnegative, we obtain for all $t\in [0,T)$,
\begin{equation*}
\begin{aligned}
   X_{t}^{\infty} & = \xi+\int_{0}^{t}\bigg[A_s X_s^{\infty}+f(s, \mathbb{E}[X_s^{\infty}| \mathcal{F}_s^{W^{0}}])-B_s^2 R_s^{-1} u^{\infty}(s,X_{s}^{\infty},\mathbb{E}[X^{\infty}_{s}|\mathcal{F}_{s}^{W^{0}}]) \\
   &\quad \quad \quad \quad -B_s h\left(s, \rho(s,-R_s^{-1} B_s u^{\infty}(s,\mathbb{E}[X^{\infty}_{s}|\mathcal{F}_{s}^{W^{0}}],\mathbb{E}[X^{\infty}_{s}|\mathcal{F}_{s}^{W^{0}}]))\right)\\
& \quad \quad \quad \quad  +b\left(s, \rho(s,-R_s^{-1} B_s u^{\infty}(s,\mathbb{E}[X^{\infty}_{s}|\mathcal{F}_{s}^{W^{0}}],\mathbb{E}[X^{\infty}_{s}|\mathcal{F}_{s}^{W^{0}}]))\right)\bigg] ds\\
& \leq \xi+ tK\xi+tK\mathbb{E}[\xi]-\frac{\delta^{2}}{K}\int_{0}^{t} \frac{C_{1}X_{s}^{\infty}}{(T-s)}ds\\
    & \leq (tK+1)\xi+tK\mathbb{E}[\xi]-\frac{C_{1}\delta^{2}X_{t}^{\infty}}{K}\int_{0}^{t} \frac{1}{(T-s)}ds.   
    \end{aligned}
\end{equation*}
Hence it holds 
\begin{equation*}
\begin{aligned}
  &  \quad X_{t}^{\infty}\left(1+\frac{C_{1}\delta^{2}}{K}\int_{0}^{t}\frac{1}{T-s}ds\right)\\
  & \leq (tK+1) \xi+tK\mathbb{E}[\xi].
  \end{aligned}
\end{equation*}
Taking the limit $t \rightarrow T$ implies that $X_{T}^{\infty} =  0$.
\end{proof}

Now we define the control 
\begin{equation}\label{infty control}
  \alpha_{t}^{\infty}:= -B_{t}R_{t}^{-1}u^{\infty}(t,X_{t}^{\infty},\mathbb{E}[X_{t}^{\infty}|\mathcal{F}_{t}^{W^{0}}])-h\left(t,\rho(t,-R_{t}^{-1}B_{t}u^{\infty}(t,X^{\infty}_{t},\mathbb{E}[X_{t}^{\infty}|\mathcal{F}_{t}^{W^{0}}]))\right),t\in [0,T), \alpha^{\infty}_{T} := 0.
\end{equation}

We note that by lemma \ref{xT=0}, we obtain for control $\alpha^{\infty}$, $X^{\infty}_{T} = 0$. It remains to show that it indeed minimize the cost functional.\\
\begin{proof}[Proof of Theorem \ref{main result}]
First we denote
$$\mathcal{V}^{L}(\xi):=\underset{\alpha \in \mathcal{A}(0,T)}{\operatorname{essinf}} \mathcal{J}^{L}\left(\alpha\right) ,\quad \mathcal{V}(\xi):=\underset{\alpha \in \mathcal{A}^{0}(0,T)}{\operatorname{essinf}} \mathcal{J}\left(\alpha\right).
$$
 Since the family of $\mathcal{J}^{L}$ is non-decreasing in $L$, we obtain that the family of $\mathcal{V}^{L}$ is also non-decreasing in $L$. In particular, the limit $\mathcal{V}^{\infty}\left( \xi\right):=\lim _{L \rightarrow \infty} \mathcal{V}^L\left(\xi\right)$ exists. Since it holds that
 \begin{equation*}
\mathcal{V}^L\left(\xi\right)=\underset{\alpha \in \mathcal{A}(0,T)}{\operatorname{essinf}} \mathcal{J}^{L}\left(\alpha\right) \leq \underset{\alpha \in \mathcal{A}^{0}(0,T)}{\operatorname{essinf}}\mathcal{J}^{L}\left( \alpha\right)=\mathcal{V}\left(\xi\right).
 \end{equation*}
 Therefore, we obtain
  \begin{equation}\label{v_infty}
   \mathcal{V}^{\infty}\left(\xi\right) \leq \mathcal{V}\left(\xi\right).
  \end{equation}
 Next, let us as above denote for every $L \in(0, \infty)$ by $\left(X^L, Y^L, Z^L\right)$ the unique solution to the FBSDE \eqref{conditional mfFBSDE}. The nonnegativity of $LX_{T}^{2}$ implies that
\begin{equation}\label{inequlity}
\mathcal{V}^L\left(\xi\right) \geq  \frac{1}{2}\mathbb{E}\left[\int_{0}^T Q_{s}(X_{s}^{L}+l(s,\mathbb{E}[X_{s}^{L}|\mathcal{F}_{s}^{W^{0}}]))^{2}+R_{s}(\alpha^{L}_{s}+h(s,\mathbb{E}[\alpha^{L}_{s}|\mathcal{F}_{s}^{W^{0}}]))^{2}d s \right].
 \end{equation}
 Note that $X^L$ converges to $X^{\infty}$ a.s. for $L \rightarrow \infty$ by definition. And $\alpha^L$ converges for $L \rightarrow \infty$ to $\alpha^{\infty}$ a.s. as well. Nonnegativity, Fatou's lemma and \eqref{inequlity} then imply that
 \begin{equation*}
\mathcal{V}^{\infty}\left(\xi\right) \geq \frac{1}{2}\mathbb{E}\left[\int_{0}^T Q_{s}(X_{s}^{\infty}+l(s,\mathbb{E}[X_{s}^{\infty}|\mathcal{F}_{s}^{W^{0}}]))^{2}+R_{s}(\alpha^{\infty}_{s}+h(s,\mathbb{E}[\alpha^{\infty}_{s}|\mathcal{F}_{s}^{W^{0}}]))^{2}d s \right]  \geq  \mathcal{V}\left(\xi\right).
 \end{equation*}
Combined with \eqref{v_infty}, we obtain
\begin{equation*}
\mathcal{V}\left(\xi\right)= \mathcal{V}^{\infty}\left(\xi\right)= \frac{1}{2}\mathbb{E}\left[\int_{0}^T Q_{s}(X_{s}^{\infty}+l(\mathbb{E}[X_{s}^{\infty}|\mathcal{F}_{s}^{W^{0}}]))^{2}+R_{s}(\alpha^{\infty}_{s}+h(\mathbb{E}[\alpha^{\infty}_{s}|\mathcal{F}_{s}^{W^{0}}]))^{2}d s \right].
 \end{equation*}
\end{proof}
\subsection{The FBSDE associated to problem (C-MF)}\label{3.3}
We have shown that the solution component $X^L$ of the FBSDE \eqref{conditional mfFBSDE} converges to $X^{\infty}$ as $L \rightarrow \infty$ and the associated decoupling field $u^L$ converges to $u^{\infty}$. This allows to define a process $Y^{\infty}$ via $Y^{\infty}=u^{\infty}\left(\cdot, X^{\infty},\mathbb{E}[X^{\infty}|\mathcal{F}^{W^{0}}]\right)$. The next result shows that also the solution components $Z^L,Z^{0,L}$ have  limits $Z^{\infty},Z^{0,\infty}$, respectively and that the process $\left(X^{\infty}, Y^{\infty}, Z^{\infty},Z^{0,\infty}\right)$ satisfies a new coupled conditional mean field FBSDE.  In contrast to \eqref{conditional mfFBSDE}, there are two constraints imposed on the forward process $X^{\infty}$, while $Y^{\infty}$ is not required to satisfy any boundary conditions.
\begin{Theorem}\label{existence}
For every $L\in(0,\infty)$, let $(X^{L},Y^{L},Z^{L},Z^{0,L})$ be the solution of the conditional mean field FBSDE \eqref{conditional mfFBSDE} and $u^{L}$ be the associated decoupling field. Let $u^{\infty}$ and $X^{\infty}$ be the limits of $u^{L}$ and $X^{L}$ as $L\rightarrow \infty$ defined by \eqref{u infty defination} and \eqref{xinfty}. Let $Y^{\infty}: \Omega \times[0, T) \rightarrow \mathbb{R}$ satisfy for all $t \in[0, T)$ a.s. that $Y_t^{\infty}=u^{\infty}\left(t, X_t^{\infty},\mathbb{E}[X_t^{\infty}|\mathcal{F}^{W^{0}}_{t}]\right)$. Then the sequence $Z^L$ converges in $L^2_{\mathbb{F}}\left([0, T);\mathbb{R}^d\right)$ to $Z^{\infty}$, the sequence $Z^{0,L}$ converges in $L^2_{\mathbb{F}}\left([0, T);\mathbb{R}^d\right)$  to $Z^{0,\infty}$, and the process $\left(X^{\infty}, Y^{\infty}, Z^{\infty},Z^{0,\infty}\right)$ satisfies for all $0 \leq t \leq r<T$ a.s. that
\begin{equation}\label{conditional constrainted mfFBSDE}
\left\{\begin{aligned}
X_{t}^{\infty} &=  \int_{0}^{t}{\left[A_s X_{s}^{\infty}-B_s^2 R_s^{-1} Y_{s}^{\infty}-B_s h\left(s, \rho(s,-R_s^{-1} B_s \mathbb{E}[Y_{s}^{\infty} |\mathcal{F}_s^{W^{0}}])\right)\right.} \\
& \left.\quad +f\left(s, \mathbb{E}[X_{s}^{\infty} |\mathcal{F}_s^{W^{0}}]\right)+b\left(s, \rho(s,-R_s^{-1} B_s \mathbb{E}[Y_{s}^{\infty}|\mathcal{F}_s^{W^{0}}])\right)\right] ds,\quad X_{T}^{\infty} = 0 \\
Y_{t}^{\infty} &=Y_{r}^{\infty}+\int_{t}^{r}\left[A_sY_s^{\infty}+Q_sX_s^{\infty}+Q_s l\left(s, \mathbb{E}[X_{s}^{\infty} |\mathcal{F}_s^{W^{0}}]\right)\right] d s-\int_{t}^{r}Z_{s}^{\infty}d W_s-\int_{t}^{r}Z_s^{0, \infty} dW_s^0.
\end{aligned}\right.
\end{equation}
\end{Theorem}
\begin{proof}
By lemma \ref{xT=0} and the definition of $Y^{\infty}$, we obtain the forward equation in \eqref{conditional constrainted mfFBSDE} holds. It follows from Theorem \ref{LQ} that for any $L>0$, it holds
\begin{equation}\label{yL}
    Y_{t}^{L} =Y_{r}^{L}+\int_{t}^{r}\left[A_sY_s^{L}+Q_sX_s^{L}+Q_s l\left(s, \mathbb{E}[X_{s}^{L} |\mathcal{F}_s^{W^{0}}]\right)\right] d s-\int_{t}^{r}Z_{s}^{L}d W_s-\int_{t}^{r}Z_s^{0,L} dW_s^0. 
\end{equation}
Applying It\^{o}'s formula for $L\leq L^{\prime}$, we have
\begin{equation}\label{z}
    \begin{aligned}
        &\mathbb{E}\int_{t}^{r}|Z_{s}^{L}-Z_{s}^{L^{\prime}}|^{2}ds+ \mathbb{E}\int_{t}^{r}|Z_{s}^{0,L}-Z_{s}^{0,L^{\prime}}|^{2}ds\\ 
        &\leq \mathbb{E}|Y_{r}^{L}-Y_{r}^{L^{\prime}}|^{2}+2C\mathbb{E}\int_{t}^{r}|Y_{s}^{L}-Y_{s}^{L^{\prime}}|\left(|Y_{s}^{L}-Y_{s}^{L^{\prime}}|+|X_{s}^{L}-X_{s}^{L^{\prime}}|+\left|\mathbb{E}[X_{s}^{L}|\mathcal{F}_{s}^{W^{0}}]-\mathbb{E}[X_{s}^{L^{\prime}}|\mathcal{F}_{s}^{W^{0}}]\right|\right)\\
   & \leq \mathbb{E}|Y_{r}^{L}-Y_{r}^{L^{\prime}}|^{2}+4C\mathbb{E}\int_{t}^{r}\left(|Y_{s}^{L}-Y_{s}^{L^{\prime}}|^{2}+|X_{s}^{L}-X_{s}^{L^{\prime}}|^{2}+|\mathbb{E}[X_{s}^{L}|\mathcal{F}_{s}^{W^{0}}]-\mathbb{E}[X_{s}^{L^{\prime}}|\mathcal{F}_{s}^{W^{0}}]|^{2}\right).
    \end{aligned}
\end{equation}

Lemma \ref{Xinfty} implies that $X^{L}$  converges in the space $L^2_{\mathbb{F}}\left([0, T);\mathbb{R}\right)$. Moreover, by lemma \ref{uniformly bounded}, we have for all $s \in[0, r]$ and $L>0$ that $\max \left(\left|Y_s^L\right|,\left|Y_s^{\infty}\right|\right) $ is uniformly bounded. By the dominated convergence theorem, the sequence $Y^L$ also converges in the space $L^2_{\mathbb{F}}\left([0, T); \mathbb{R}\right)$. Combining this with \eqref{z}, we deduce that $Z^L$ and $Z^{0,L}$ are Cauchy sequences in $L^2_{\mathbb{F}}\left(\left(0, r\right), \mathbb{R}^{d}\right)$, and thus converges to $Z^{\infty}$ and $Z^{0,\infty}$, respectively. 
Therefore, we can take the limit $L \rightarrow \infty$ in \eqref{yL} to obtain
\begin{equation*}
Y_{t}^{\infty} =Y_{r}^{\infty}+\int_{t}^{r}\left[A_sY_s^{\infty}+Q_sX_s^{\infty}+Q_s l(s, \mathbb{E}[X_{s}^{\infty} |\mathcal{F}_s^{W^{0}}])\right] d s-\int_{t}^{r}Z_{s}^{\infty}d W_s-\int_{t}^{r}Z_s^{0, \infty} dW_s^0.    
\end{equation*}
This completes the proof.
\end{proof}

Next, we characterize $(X^{\infty},Y^{\infty},Z^{\infty},Z^{0,\infty})$ as the unique solution of \eqref{conditional constrainted mfFBSDE} within a class of processes satisfying suitable assumptions.
\begin{Theorem}
 The process $\left(X^{\infty}, Y^{\infty}, Z^{\infty}, Z^{0,\infty}\right)$ constructed in Theorem \ref{existence} is the unique solution of \eqref{conditional constrainted mfFBSDE} satisfying $(X^{\infty}, Y^{\infty}, Z^{\infty},Z^{0,\infty}) \in L^{2}_{\mathbb{F}}([0,T);\mathbb{R})\times L^{2}_{\mathbb{F}}([0,T);\mathbb{R})\times L^{2}_{\mathbb{F}}([0,T);\mathbb{R}^{d})\times L^{2}_{\mathbb{F}}([0,T);\mathbb{R}^{d})$ and $X^{\infty}, Y^{\infty}\geq 0$.
\end{Theorem}
\begin{proof}
Let $(X,Y,Z,Z^{0})$ be another solution of \eqref{conditional constrainted mfFBSDE}. Our objective is to prove that $(X,Y,Z,Z^{0})$ coincides with $(X^{\infty},Y^{\infty},Z^{\infty},Z^{0,\infty})$. To this end, we first show 
\begin{equation}\label{1}
    \lim_{t\rightarrow T} \mathbb{E}\left[(Y^{\infty}_{t}-Y_{t})(X^{\infty}_{t}-X_{t})\right] = 0,
\end{equation}
and
\begin{equation}\label{2}
    \lim_{t\rightarrow T} \mathbb{E}\left[(\mathbb{E}[Y^{\infty}_{t}|\mathcal{F}_{t}^{W^{0}}]-\mathbb{E}[Y_{t}|\mathcal{F}_{t}^{W^{0}}])(\mathbb{E}[X^{\infty}_{t}|\mathcal{F}_{t}^{W^{0}}]-\mathbb{E}[X_{t}|\mathcal{F}_{t}^{W^{0}}])\right] = 0.
\end{equation}
For any $t\in[0,T)$, we have
\begin{equation*}
\begin{aligned}
     X_{t} &\geq  X_{t}-X_{r}\\
     &= \int_{t}^{r}\left[-A_{s}X_{s}-f(\mathbb{E}[s,Y_{s}|\mathcal{F}_s^{W^{0}}])+B_s^2 R_s^{-1}Y_{s} +B_s h\left(s, \rho(s,-R_s^{-1} B_s \mathbb{E}[Y_{s} |\mathcal{F}_s^{W^{0}}])\right)\right.\\
     &\quad\left.-b\left(s, \rho(s,-R_s^{-1} B_s \mathbb{E}[Y_{s}|\mathcal{F}_s^{W^{0}}])\right)\right]ds\\
     &\geq  \int_{t}^{r} \frac{\delta^{2}}{K}Y_{s}ds.
\end{aligned}
\end{equation*}
Consequently, with the product rule we obtain for all $t \in\left[0, T\right)$,
$$
\begin{aligned}
(T-t) Y_t &=  (T-r)Y_r+\int_t^r Y_s d s  +\int_t^r(T-s) \left[A_sY_s+Q_sX_s+Q_s l\left(s, \mathbb{E}[X_{s} |\mathcal{F}_s^{W^{0}}]\right)\right] d s \\
&\quad -\int_t^r(T-s) Z_s d W_s-\int_t^r(T-s) Z^{0}_s d W_s^{0} \\
&\leq  (T-r) Y_r+\frac{K}{\delta^{2}}X_{t} +\int_t^r(T-s) \left[A_sY_s+Q_sX_s+Q_s l\left(s, \mathbb{E}[X_{s} |\mathcal{F}_s^{W^{0}}]\right)\right]  d s\\
&\quad  -\int_t^r(T-s) Z_s d W_s-\int_t^r(T-s) Z^{0}_s d W_s^{0} .
\end{aligned}
$$
Taking conditional expectations on both side, we obtain,
$$
\begin{aligned}
(T-t) Y_t \leq & \mathbb{E}\left[(T-r) Y_r | \mathcal{F}_t\right]+\frac{K}{\delta^{2}} X_{t} \\
& +\mathbb{E}\left[\int_t^r(T-s) \left[A_sY_s+Q_sX_s+Q_s l\left(s, \mathbb{E}[X_{s} |\mathcal{F}_s^{W^{0}}]\right)\right]d s | \mathcal{F}_t\right] .
\end{aligned}
$$
The nonnegtivity of $Q_sX_s+Q_s l\left(s, \mathbb{E}[X_{s} |\mathcal{F}_s^{W^{0}}]\right)$ implies that $e^{A_{s}}Y_{s}$ is a nonnegative supermartingale on $[0,T)$, and hence $0\leq \mathbb{E}[e^{A_{r}}Y_{r}|\mathcal{F}_{t}]\leq e^{A_{t}}Y_{t}$.
 Consequently, $\lim _{r \rightarrow T} \mathbb{E}\left[(T-r) Y_r | \mathcal{F}_t\right]=0$. Thus, by letting $r \uparrow T$, we obtain for all $t \in\left[0, T\right)$,
 $$
\begin{aligned}
(T-t) Y_t  & \leq \frac{K}{\delta^{2}} X_{t} +\mathbb{E}\left[\int_t^r(T-s) \left[A_sY_s+Q_sX_s+Q_s l\left(s, \mathbb{E}[X_{s} |\mathcal{F}_s^{W^{0}}]\right)\right]d s | \mathcal{F}_t\right]\\
&\leq \frac{K}{\delta^{2}} X_{t} +(T-t)\mathbb{E}\left[\int_t^T \left[A_sY_s+Q_sX_s+Q_s l\left(s, \mathbb{E}[X_{s} |\mathcal{F}_s^{W^{0}}]\right)\right]d s | \mathcal{F}_t\right].
\end{aligned}
$$
Next, it follows from H{\"o}lder's inequality that for all $t \in\left[0, T\right)$,
\begin{equation*}
    \begin{aligned}
\mathbb{E}[X_t^{\infty} X_t]& \leq \sqrt{\mathbb{E}[\left(X_t^{\infty}\right)^2] \mathbb{E}\left[X_t^2\right]}\\
&\leq(T-t) \sqrt{\mathbb{E}\left[\int_t^T\left(B^{\infty}(r,X^{\infty}_{r},Y^{\infty}_{r})\right)^2 d r\right] \mathbb{E}\left[\int_t^T\left(B(r,X_{r},Y_{r})\right)^2 d r\right]},  \end{aligned}
\end{equation*}
where we denote the drift coefficients of $X^{\infty}
$ and $X$ as $B^{\infty}$ and $B$, respectively.
Moreover, $X^{\infty},X,Y^{\infty},Y\in L^{2}_{\mathbb{F}}([0,T);\mathbb{R})$, $B^{\infty}(t,0,0)=0$, $B(t,0,0)=0$ imply that
 $B, B^{\infty} \in L^2_{\mathbb{F}}\left(\left(0, T\right); \mathbb{R}\right)$. Then
$$
\lim _{t \rightarrow T} \frac{1}{T-t} \mathbb{E}\left[X_t^{\infty} X_t\right]=0 .
$$
Then, we have
$$
\begin{aligned}
\mathbb{E}\left[Y_t X_t^{\infty}\right] & \leq \frac{K}{\delta^{2}}\mathbb{E}\left[ \frac{X_t X_t^{\infty}}{T-t}\right]+\mathbb{E}\left[X_t^{\infty} \mathbb{E}\left[\int_t^T \left[A_sY_s+Q_sX_s+Q_s l\left(s, \mathbb{E}[X_{s} |\mathcal{F}_s^{W^{0}}]\right)\right] d s | \mathcal{F}_t\right]\right] .
\end{aligned}
$$
Therefore, we obtain $\lim _{t \rightarrow T} \mathbb{E}\left[Y_t X_t^{\infty}\right]=0$.
Similarly, one can show $\lim _{t \rightarrow T}\mathbb{E}\left[Y_t X_t\right]=\lim _{t \rightarrow T} \mathbb{E}\left[Y_t^{\infty} X_t^{\infty}\right]=\lim _{t \rightarrow T}\mathbb{E}\left[Y_t^{\infty} X_t\right]=0$, which entails \eqref{1}. And we can obtain \eqref{2} similarly.\\ 
Now, applying It\^{o}'s formula to $(\mathbb{E}[Y^{\infty}_{t}|\mathcal{F}_{t}^{W^{0}}]-\mathbb{E}[Y_{t}|\mathcal{F}_{t}^{W^{0}}])(e^{K(\omega)t}(\mathbb{E}[X^{\infty}_{t}|\mathcal{F}_{t}^{W^{0}}]-\mathbb{E}[X_{t}|\mathcal{F}_{t}^{W^{0}}]))$, 
we get
\begin{equation*}
\begin{aligned}
& \mathbb{E}[Y^{\infty}_{t}|\mathcal{F}_{t}^{W^{0}}]-\mathbb{E}[Y_{t}|\mathcal{F}_{t}^{W^{0}}])(e^{K(\omega)t}(\mathbb{E}[X^{\infty}_{t}|\mathcal{F}_{t}^{W^{0}}]-\mathbb{E}[X_{t}|\mathcal{F}_{t}^{W^{0}}])) \\
 & = \int_{0}^{t}-\Big(A_s\mathbb{E}[Y_s^{\infty}|\mathcal{F}^{W^{0}}_{s}]+Q_s\mathbb{E}[X_s^{\infty}|\mathcal{F}^{W^{0}}_{s}]+Q_s l(s, \mathbb{E}[X_{s}^{\infty} |\mathcal{F}_s^{W^{0}}])\\
 &\quad-A_s\mathbb{E}[Y_s|\mathcal{F}^{W^{0}}_{s}]-Q_s\mathbb{E}[X_s^{\infty}|\mathcal{F}^{W^{0}}_{s}]-Q_s l(s, \mathbb{E}[X_{s} |\mathcal{F}_s^{W^{0}}]))\Big)\Big(e^{K(\omega)s}(\mathbb{E}[X_s^{\infty}|\mathcal{F}^{W^{0}}_{s}]-\mathbb{E}[X_s|\mathcal{F}^{W^{0}}_{s}])\Big)\\
 &\quad +e^{K(\omega)s}\Bigg(A_s \mathbb{E}[X_s^{\infty}|\mathcal{F}^{W^{0}}_{s}]-B_s^2 R_s^{-1}\mathbb{E}[Y_s^{\infty}|\mathcal{F}^{W^{0}}_{s}]-B_s h\left(s, \rho(s,-R_s^{-1} B_s \mathbb{E}[Y_{s}^{\infty} |\mathcal{F}_s^{W^{0}}])\right) \\
& \quad +b\left(s, \rho(s,-R_s^{-1} B_s \mathbb{E}[Y_{s}^{\infty}|\mathcal{F}_s^{W^{0}}])\right)-A_{s}\mathbb{E}[X_s|\mathcal{F}^{W^{0}}_{s}]+B_s^2 R_s^{-1}\mathbb{E}[Y_s|\mathcal{F}^{W^{0}}_{s}] \\
     &\quad+B_s h\left(s, \rho(s,-R_s^{-1} B_s \mathbb{E}[Y_{s} |\mathcal{F}_s^{W^{0}}])\right)-b\left(s, \rho(s,-R_s^{-1} B_s \mathbb{E}[Y_{s}|\mathcal{F}_s^{W^{0}}])\right)\\
     &\quad+\frac{f(s,\mathbb{E}[X_{s}^{\infty}|\mathcal{F}^{W^{0}}_{s}])-f(s,\mathbb{E}[X_{s}|\mathcal{F}^{W^{0}}_{s}])}{\mathbb{E}[X_s^{\infty}|\mathcal{F}^{W^{0}}_{s}]-\mathbb{E}[X_s|\mathcal{F}^{W^{0}}_{s}]}(\mathbb{E}[X_s^{\infty}|\mathcal{F}^{W^{0}}_{s}]-\mathbb{E}[X_s|\mathcal{F}^{W^{0}}_{s}])\\
     &\quad+K(\omega)(\mathbb{E}[X_s^{\infty}|\mathcal{F}^{W^{0}}_{s}]-\mathbb{E}[X_s|\mathcal{F}^{W^{0}}_{s}])\Bigg)\Big(\mathbb{E}[Y_s^{\infty}|\mathcal{F}^{W^{0}}_{s}]-\mathbb{E}[Y_s|\mathcal{F}^{W^{0}}_{s}]\Big)ds.
    \end{aligned}
\end{equation*}
Let 
$$
K(\omega)=\frac{f(s,\mathbb{E}[X_{s}^{\infty}|\mathcal{F}^{W^{0}}_{s}])-f(s,\mathbb{E}[X_{s}|\mathcal{F}^{W^{0}}_{s}])}{\mathbb{E}[X_s^{\infty}|\mathcal{F}^{W^{0}}_{s}]-\mathbb{E}[X_s|\mathcal{F}^{W^{0}}_{s}]}.\footnote{we will take the convention that
	$
	\frac{\varphi(x)-\varphi(x)}{x-x}:=\lim _{\tilde{x} \rightarrow x} \frac{\varphi(\tilde{x})-\varphi(x)}{\tilde{x}-x}$.
}
$$
Then, we have
\begin{equation*}
\begin{aligned}
& \mathbb{E}[Y^{\infty}_{t}|\mathcal{F}_{t}^{W^{0}}]-\mathbb{E}[Y_{t}|\mathcal{F}_{t}^{W^{0}}])(e^{K(\omega)t}(\mathbb{E}[X^{\infty}_{t}|\mathcal{F}_{t}^{W^{0}}]-\mathbb{E}[X_{t}|\mathcal{F}_{t}^{W^{0}}])) \\
	 & =\int_{0}^{t}-Q_s\left(\mathbb{E}[X_s^{\infty}|\mathcal{F}^{W^{0}}_{s}]-\mathbb{E}[X_s|\mathcal{F}^{W^{0}}_{s}]\right)^{2}-B_s^2 R_s^{-1}\Big(\mathbb{E}[Y_s^{\infty}|\mathcal{F}^{W^{0}}_{s}]-\mathbb{E}[Y_s|\mathcal{F}^{W^{0}}_{s}]\Big)^{2}\\
	&\quad-Q_s \left(l(s, \mathbb{E}[X_{s}^{\infty} |\mathcal{F}_s^{W^{0}}])-l(s, \mathbb{E}[X_{s} |\mathcal{F}_s^{W^{0}}]))\right)\Big(e^{K(\omega)s}(\mathbb{E}[X_s^{\infty}|\mathcal{F}^{W^{0}}_{s}]-\mathbb{E}[X_s|\mathcal{F}^{W^{0}}_{s}])\Big)\\
	&\quad \Big[-B_sh\left(s, \rho(s,-R_s^{-1} B_s \mathbb{E}[Y_{s}^{\infty} |\mathcal{F}_s^{W^{0}}])\right)+b\left(s, \rho(s,-R_s^{-1} B_s \mathbb{E}[Y_{s}^{\infty}|\mathcal{F}_s^{W^{0}}])\right) \\
	& \quad +B_{s}h\left(s, \rho(s,-R_s^{-1} B_s \mathbb{E}[Y_{s} |\mathcal{F}_s^{W^{0}}])\right)-b\left(s, \rho(s,-R_s^{-1} B_s \mathbb{E}[Y_{s}|\mathcal{F}_s^{W^{0}}])\right)\Big]\Big(e^{K(\omega)s}(\mathbb{E}[Y_s^{\infty}|\mathcal{F}^{W^{0}}_{s}]-\mathbb{E}[Y_s|\mathcal{F}^{W^{0}}_{s}])\Big)ds\\
	&\leq \int_{0}^{t} -C\left(\left|\mathbb{E}[X_{s}^{\infty}|\mathcal{F}_{s}^{W^{0}}]-\mathbb{E}[X_{s}|\mathcal{F}_{s}^{W^{0}}]\right|^{2}+\left|\mathbb{E}[Y_{s}^{\infty}|\mathcal{F}_{s}^{W^{0}}]-\mathbb{E}[Y_{s}|\mathcal{F}_{s}^{W^{0}}]\right|^{2}\right)ds.
\end{aligned}
\end{equation*}
Let $t\rightarrow T$, we obtain
$$
0 \leq \int_{0}^{T} -C\left(\left|\mathbb{E}[X_{s}^{\infty}|\mathcal{F}_{s}^{W^{0}}]-\mathbb{E}[X_{s}|\mathcal{F}_{s}^{W^{0}}]\right|^{2}+\left|\mathbb{E}[Y_{s}^{\infty}|\mathcal{F}_{s}^{W^{0}}]-\mathbb{E}[Y_{s}|\mathcal{F}_{s}^{W^{0}}]\right|^{2}\right)ds,
$$
which implies that
\begin{equation}\label{conditional exp}
    \mathbb{E}[X_{s}^{\infty}|\mathcal{F}_{s}^{W^{0}}]=\mathbb{E}[X_{s}|\mathcal{F}_{s}^{W^{0}}],\quad \mathbb{E}[Y_{s}^{\infty}|\mathcal{F}_{s}^{W^{0}}]=\mathbb{E}[Y_{s}|\mathcal{F}_{s}^{W^{0}}].
\end{equation} 
Then, applying It\^{o}'s formula to $(Y_{t}^{\infty}-Y_{t})(X_{t}^{\infty}-X_{t})$ and using \eqref{conditional exp}, we have
\begin{equation*}
\begin{aligned}
    (Y_{t}^{\infty}-Y_{t})(X_{t}^{\infty}-X_{t})& = \int_{0}^{t} \left[-Q_{s}(X_{s}^{\infty}-X_{s})^{2}-B^{2}_{s}R^{-1}_{s}(Y_{s}^{\infty}-Y_{s})^{2}\right]ds\\
    & \leq \int_{0}^{t}-C\left[(X_{s}^{\infty}-X_{s})^{2}+(Y_{s}^{\infty}-Y_{s})^{2}\right]ds.
\end{aligned}
\end{equation*}
Let $t\rightarrow T$, we can get $Y^{\infty}=Y$ and $X^{\infty}=X$, and hence also $Z^{\infty}=Z$ and $Z^{0,\infty}=Z^{0}$.
\end{proof}
\bibliographystyle{siam}
\bibliography{bib}
\end{document}